\documentclass[11pt]{amsart}
\usepackage[titletoc]{appendix}

\calclayout

\usepackage{amssymb,amsmath,amsthm}
\usepackage{enumerate}
\usepackage{mathtools}

\usepackage[utf8]{inputenc}
\usepackage{color}
\usepackage[colorlinks,linkcolor=blue,citecolor=blue]{hyperref}

\def\AA{\mathord{\mathbf{A}}}
\def\CC{\mathord{\mathbf{C}}}

\def\QQ{\mathord{\mathbf{Q}}}

\def\ZZ{\mathord{\mathbf{Z}}}

\def\transp{^{\mathord{\textsf{T}}}}

\def\opp{{\rm op}}

\def\Sets{\textsf{Set}}
\def\Sch{\textsf{Sch}}

\def\Sym{\mathop{\rm Sym}}
\def\Sp{\mathop{\rm Sp}}

\def\Aut{\mathop{\rm Aut}}

\def\GL{\mathop{\rm GL}}

\def\rank{\mathop{\rm rank}}

\def\Lie{\mathop{\rm Lie}}

\def\Spec{\mathop{\rm Spec}}

\def\Pic{\mathop{\rm Pic }}

\def\dlog{\mathop{\rm dlog}}
\def\Et{\mathop{\text{Ét}}}

\def\trdeg{\text{\rm trdeg}}

\def\id{\text{\rm id}}
\def\an{\text{\rm an}}
\def\dR{\text{\rm dR}}

\def\et{\text{\rm ét}}

\def\et{\text{\rm ét}}

\def\defeq{\vcentcolon=}
\def\eqdef{=\vcentcolon}
\def\tensor{\otimes}

\def\to{\longrightarrow}
\def\mapsto{\longmapsto}

\newtheorem{theorem}{Theorem}[section]
\newtheorem{prop}[theorem]{Proposition}
\newtheorem{lemma}[theorem]{Lemma}

\newtheorem{coro}[theorem]{Corollary}

\theoremstyle{definition}

\newtheorem{defi}[theorem]{Definition}
\newtheorem{obs}[theorem]{Remark}

\numberwithin{equation}{section}
\title[Higher Ramanujan equations I]{Higher Ramanujan equations I: \\ moduli stacks of abelian varieties and higher Ramanujan vector fields}
\author{Tiago J. Fonseca}

\address{Laboratoire de Mathématiques d'Orsay, Univ. Paris-Sud, CNRS, Université
Paris-Saclay, 91405 Orsay, France.}
\email{tiago.jardim-da-fonseca@math.u-psud.fr}

\date{\today}

\begin{document}
\maketitle

\begin{abstract}
We describe a higher dimensional generalization of Ramanujan's differential equations satisfied by the Eisenstein series $E_2$, $E_4$, and $E_6$. This will be obtained geometrically as follows. For every integer $g\ge 1$, we construct a moduli stack $\mathcal{B}_g$ over $\ZZ$ classifying principally polarized abelian varieties of dimension $g$ equipped with a suitable additional structure:  a \emph{symplectic-Hodge basis} of its first algebraic de Rham cohomology. We prove that $\mathcal{B}_g$ is a smooth Deligne-Mumford stack over $\ZZ$ of relative dimension $2g^2 + g$ and that $\mathcal{B}_g\tensor \ZZ[1/2]$ is representable by a smooth quasi-projective scheme over $\ZZ[1/2]$. Our main result is a description of the tangent bundle $T_{\mathcal{B}_g/\ZZ}$ in terms of the cohomology of the universal abelian scheme over the moduli stack of principally polarized abelian varieties $\mathcal{A}_g$. We derive from this description a family of $g(g+1)/2$ commuting vector fields $(v_{ij})_{1\le i \le j \le g}$ on $\mathcal{B}_g$; these are the \emph{higher Ramanujan vector fields}. In the case $g=1$, we show that $v_{11}$ coincides with the vector field associated to the classical Ramanujan equations.

This geometric framework taking account of integrality issues is mainly motivated by questions in transcendental number theory. In the upcoming second part of this work, we shall relate the values of a particular analytic solution to the differential equations defined by $v_{ij}$ with Grothendieck's periods conjecture on abelian varieties.
\end{abstract}

% \begin{abstract}
% For every integer $g\ge 1$ we construct and study a moduli stack $\mathcal{B}_g$ over $\ZZ$ classifying principally polarized abelian varieties of dimension $g$ equipped with a suitable additional structure; namely, a \emph{symplectic-Hodge basis} of its first algebraic de Rham cohomology. We prove that $\mathcal{B}_g$ is a smooth Deligne-Mumford stack over $\ZZ$ of relative dimension $2g^2 + g$ and that $\mathcal{B}_g\tensor \ZZ[1/2]$ is representable by a smooth quasi-projective scheme over $\ZZ[1/2]$. We describe the tangent bundle $T_{\mathcal{B}_g/\ZZ}$ in terms of the cohomology of the universal abelian scheme over the moduli stack of principally polarized abelian varieties $\mathcal{A}_g$, and we derive from this description a family of $g(g+1)/2$ commuting vector fields $v_{ij} \in \Gamma(\mathcal{B}_g,T_{\mathcal{B}_g/\ZZ})$. In the case $g=1$, $v_{11}$ boils down to the vector field associated to the classical Ramanujan equations on Eisenstein series.
% \end{abstract}

\tableofcontents

\section{Introduction}

Consider the classical normalized Eisenstein series in $\ZZ[\![q]\!]$
\begin{align*}
E_{2}(q) = 1 - 24 \sum_{n=1}^{\infty}\frac{n q^n}{1-q^n}\text{, }\ \ E_{4}(q) = 1+240 \sum_{n=1}^{\infty}\frac{n^3 q^n}{1-q^n}\text{, } \ \ 
E_6(q) = 1-504 \sum_{n=1}^{\infty}\frac{n^5 q^n}{1-q^n}
\end{align*}
and let $\theta \defeq q\frac{d}{dq}$. In 1916 \cite{ramanujan16} Ramanujan proved that these formal series satisfy the system of algebraic differential equations
\begin{align} \label{rameq} \tag{R}
\theta E_2 = \frac{E_2^2 - E_4}{12}\text{, }\ \ \theta E_3 = \frac{E_2E_4 - E_6}{3}\text{, } \ \ \theta E_4 = \frac{E_2E_6 - E_4^2}{2}
\text{.}
\end{align}
The study of equivalent forms of such differential equations actually predates Ramanujan. To the best of our knowledge, Jacobi was the first to prove in 1848 \cite{jacobi48} that his \emph{Thetanullwerte} satisfy a third order algebraic differential equation. In 1881 \cite{halphen81} Halphen found a simpler description of Jacobi's equation by considering logarithmic derivatives. Further, in 1911 \cite{chazy11} Chazy considered a third order differential equation\footnote{In Chazy's original notation (cf. \cite{chazy11} (4)) the equation he considered is written as $y''' = 2yy'' -3(y')^2$. If derivatives in this equation are with respect to a variable $t$, equation (\ref{chazyeq}) is obtained from this one by the change of variables $q=e^{2t}$.} satisfied by the Eisenstein series $E_2$:
\begin{align*} \label{chazyeq} \tag{C}
\theta^3E_2 = E_2\theta^2E_2 - \frac{3}{2}(\theta E_2)^2\text{.}
\end{align*}
We refer to \cite{ohyama95} for a thorough study of Jacobi's, Halphen's, and Chazy's equations, and the relations between them. We point out that Ramanujan's and Chazy's equations concern level 1 (quasi-)modular forms, whereas the equations of Jacobi and Halphen involve level 2 (quasi-)modular forms.

 A higher dimensional generalization of Jacobi's equation concerning \emph{Thetanullwerte} of complex abelian varieties of dimension $2$ was first given by Ohyama \cite{ohyama96} in 1996, and for any dimension by Zudilin \cite{zudilin00} in 2000 (see also \cite{BZ03}).

This paper, and its sequel, grew out from our attempt to obtain a more conceptual understanding of the Ramanujan equations and their higher dimensional extensions. This could possibly shed some light on their arithmetical and geometric properties. An important motivation for this program is the central role of the original Ramanujan equations (\ref{rameq}) and of the integrality properties of the series $E_2$, $E_4$, and $E_6$, in Nesterenko's celebrated result on the transcendence of their values, when regarded as holomorphic functions on the complex unit disc $D = \{z \in \CC \mid |z|<1\}$:

\begin{theorem}[Nesterenko \cite{nesterenko96} 1996]
For every $q \in D\setminus\{0\}$,
\begin{align*}
\trdeg_{\QQ}\QQ(q,E_2(q),E_4(q),E_6(q)) \ge 3\text{.}
\end{align*}
\end{theorem}

In contrast with the concrete methods of Ohyama and Zudilin based on theta functions, our geometric approach allows us to construct by purely algebraic methods some higher dimensional avatars of the system (\ref{rameq}), involving suitable moduli spaces of abelian varieties that enjoy remarkable smoothness properties over $\ZZ$. Another important difference between our approach and that of Ohyama and Zudilin is that we work in ``level 1'', although it should be clear that we can also introduce higher level structures in the picture.

 We next explain our main results.

Fix an integer $g\ge 1$. Let $k$ be a field and $(X,\lambda)$ be a principally polarized abelian variety over $k$ of dimension $g$ (here $\lambda$ denotes a suitable isomorphism from $X$ onto the dual abelian variety $X^t$). Then the first algebraic de Rham cohomology $H^1_{\dR}(X/k)$ is a $k$-vector space of dimension $2g$ endowed with a canonical subspace $F^1(X/k)$ of dimension $g$ (given by the Hodge filtration) and a non-degenerate alternating $k$-bilinear form 
\begin{align*}
\langle \ , \ \rangle_{\lambda} : H^1_{\dR}(X/k) \times H^1_{\dR}(X/k) \to k
\end{align*}
induced by the principal polarization $\lambda$. By a \emph{symplectic-Hodge basis} of $(X,\lambda)$, we mean a basis $b=(\omega_1,\ldots,\omega_g,\eta_1,\ldots,\eta_g)$ of the $k$-vector space $H^1_{\dR}(X/k)$, such that
\begin{enumerate}
   \item each $\omega_i$ is in $F^1(X/k)$, and
   \item $b$ is symplectic with respect to $\langle \ , \ \rangle_{\lambda}$, that is, $\langle \omega_i,\omega_j \rangle_{\lambda} = \langle \eta_i,\eta_j \rangle_{\lambda} =0$ and $\langle \omega_i, \eta_j \rangle_{\lambda}=\delta_{ij}$ for every $1\le i,j\le g$. 
\end{enumerate} 
We may consider the moduli stack $\mathcal{B}_g$ classifying principally polarized abelian varieties of dimension $g$ equipped with a symplectic-Hodge basis; we prove that $\mathcal{B}_g$ is a smooth Deligne-Mumford stack over $\Spec \ZZ$ of relative dimension $2g^2 + g$. This stack is not representable by a scheme (or even an algebraic space).  However, we prove that $\mathcal{B}_g\tensor \ZZ[1/2]$ is representable by a smooth quasi-projective scheme $B_g$ over $\ZZ[1/2]$. This result relies essentially on a theorem of Oda (\cite{oda69} Corollary 5.11) relating $H^1_{\dR}(X/k)$ to the Dieudonné module associated to the $p$-torsion subscheme $X[p]$ when $k$ is a perfect field of characteristic $p$. 

The main result in this paper is a description of the tangent bundle $T_{\mathcal{B}_g/\ZZ}$ in terms of the first relative de Rham cohomology of the universal abelian scheme over the moduli stack $\mathcal{A}_g$ of principally polarized abelian varieties of dimension $g$ (see Theorem \ref{theorem1} for a precise statement). From this description, we construct a family $(v_{ij})_{1\le i \le j \le g}$ of $g(g+1)/2$ commuting vector fields over $\mathcal{B}_g$; these are the \emph{higher Ramanujan vector fields}. Concretely, if $(\omega_1,\ldots,\omega_g,\eta_1,\ldots,\eta_g)$ denotes the universal symplectic-Hodge basis over $\mathcal{B}_g$, and $\nabla$ denotes the Gauss-Manin connection on the first relative de Rham cohomology of the universal abelian scheme over $\mathcal{B}_g$, then for every $1\le i \le j \le g$ we have
\begin{enumerate}
   \item $\nabla_{v_{ij}}\omega_i = \eta_j$, $\nabla_{v_{ij}}\omega_j=\eta_i$, and $\nabla_{v_{ij}}\omega_k = 0$ for every  $k\not\in \{i,j\}$, 
  \item $\nabla_{v_{ij}}\eta_k = 0$ for every $1\le k\le g$,
\end{enumerate}
and these equations completely determine $v_{ij}$.

When $g=1$, we shall recall how $B_1$ may be identified, by means of the classical theory of elliptic curves, with an open subscheme of $\AA^3_{\ZZ[1/2]} = \Spec \ZZ[1/2,b_2,b_4,b_6]$. Under this isomorphism, the vector field $v_{11}$ gets identified with
\begin{align*}
2b_4\frac{\partial}{\partial b_2} + 3b_6\frac{\partial}{\partial b_4} + (b_2b_6-b_4^2)\frac{\partial}{\partial b_6}
\end{align*} 
which is, up to scaling, the vector field associated to Chazy's equation (\ref{chazyeq})\footnote{An integral curve of this vector field for the derivation $\theta$ is given by $q \mapsto (E_2(q),\frac{1}{2}\theta E_2(q), \frac{1}{6}\theta^2E_2(q))$.}. We also show that $B_1\tensor \ZZ[1/6]$ may be identified with the open subscheme $\Spec \ZZ[1/6,e_2,e_4,e_6,1/(e_4^3-e_3^2)]$ of $\AA^3_{\ZZ[1/6]}$, and that, under this isomorphism, the vector field $v_{11}$ gets identified with the ``original'' vector field associated to the Ramanujan equations (\ref{rameq}):
\begin{align*}
 \frac{e_2^2-e_4}{12}\frac{\partial}{\partial e_2} + \frac{e_2e_4-e_6}{3}\frac{\partial}{\partial e_4} + \frac{e_2e_6-e_4^2}{2}\frac{\partial}{\partial e_6}\text{.}
\end{align*}
A geometric description of the above vector field in terms of the universal elliptic curve and the Gauss-Manin connection on its de Rham cohomology has actually been given by Movasati in \cite{movasati08} (see also \cite{movasati12}), and this has been one of the starting points of our construction. Let us remark that this point of view was already implicitly contained in the concept of ``Serre derivative'' of modular forms (\cite{serre72} 1.4) and in its geometric interpretation given by Deligne (\cite{katz73} A1.4).

In the sequel of this paper, \emph{Higher Ramanujan equations II: periods of abelian varieties and transcendence questions}, we shall introduce analytic methods in our construction and we shall tackle some transcendence questions. We shall prove, for instance, that every leaf of the holomorphic foliation on the complex manifold $B_g(\CC)$ defined by the higher Ramanujan vector fields is Zariski-dense in $B_g$. We shall also construct a particular solution $\varphi_g$ to the differential equations defined by the higher Ramanujan vector fields that will constitute a higher dimensional generalization of the solution $q\mapsto (E_2(q),E_4(q),E_6(q))$ when $g=1$. Finally, we shall give a precise relation between the transcendence degree over $\QQ$ of values of $\varphi_g$ and Grothendieck's periods conjecture on abelian varieties.

We expect that the results in this paper, and in its sequel, might interest specialists in transcendental number theory. We have tried to keep prerequisites in abelian schemes and algebraic stacks to a minimum by recalling many notions and constructions that are well known to specialists in algebraic geometry, and by citing precise results in the (rather scarce) literature on these subjects.

\subsection{Acknowledgments}

This work is part of my PhD thesis under the supervision of Jean-Benoît Bost. I would like to thank him for suggesting me this research theme, and for his careful reading of the manuscript of this paper.

\subsection{Terminology and notations} \label{terminology}

\subsubsection{} \label{vectorbundles} By a \emph{vector bundle} over a scheme $U$ we mean a locally free sheaf $\mathcal{E}$ over $U$ of finite rank. A \emph{line bundle} is a vector bundle of rank $1$. A \emph{subbundle} of $\mathcal{E}$ is a subsheaf $\mathcal{F}$ of $\mathcal{E}$ such that $\mathcal{F}$ and $\mathcal{E}/\mathcal{F}$ are also vector bundles, that is, $\mathcal{F}$ is locally a direct factor of $\mathcal{E}$. If $\mathcal{E}$ has constant rank $r$, by a \emph{basis} of $\mathcal{E}$ over $U$ we mean an ordered family of $r$ global sections of $\mathcal{E}$ that generate this sheaf as an $\mathcal{O}_U$-module. The \emph{dual} of a vector bundle $\mathcal{E}$ is the vector bundle $\mathcal{E}^{\vee} \defeq \mathcal{H}om_{\mathcal{O}_U}(\mathcal{E},\mathcal{O}_U)$.

\subsubsection{} Let $U$ be a scheme. By an \emph{abelian scheme} over $U$, we mean a proper and smooth group scheme $p:X \to U$ over $U$ with geometrically connected fibers. The group law of $X$ over $U$ is commutative (cf. \cite{GIT94} Corollary 6.5) and will be denoted additively. A \emph{morphism of abelian schemes} over $U$ is a morphism of $U$-group schemes.

 When $p$ is projective, the relative Picard functor $\Pic_{X/U}$ is representable by a group scheme over $U$ (\cite{BLR90} Chapter 8). Then, the open group subscheme $X^t$ of $\Pic_{X/U}$, whose geometric points correspond to line bundles some power of which are algebraically equivalent to zero, is a projective abelian scheme over $U$, called the \emph{dual abelian scheme}; we denote its structural morphism by $p^t:X^t \to U$. There is a canonical biduality isomorphism $X \stackrel{\sim}{\to} X^{tt}$ (cf. \cite{BLR90} 8.4 Theorem 5). The formation of both the dual abelian scheme and the biduality isomorphism is compatible with every base change in $U$. The universal line bundle over $X\times_U X^t$, the so-called \emph{Poincaré line bundle}, will be denoted by $\mathcal{P}_{X/U}$.

A \emph{principal polarization} on a projective abelian scheme $X$ over $U$ is an isomorphism of $U$-group schemes $\lambda : X \to X^t$ satisfying the equivalent conditions (cf. \cite{GIT94} 6.2 and \cite{DP94} 1.4)
\begin{enumerate}
    \item $\lambda$ is symmetric (i.e. $\lambda=\lambda^t$ under the biduality isomorphism $X\cong X^{tt}$) and  $(\id_X,\lambda)^*\mathcal{P}_{X/U}$ is relatively ample over $U$.
   \item Étale locally over $U$, $\lambda$ is \emph{induced by a line bundle on $X$} (cf. \cite{GIT94} Definition 6.2) relatively ample over $U$. 
\end{enumerate}
A \emph{principally polarized abelian scheme} over $U$ is a couple $(X,\lambda)$, where $X$ is a projective abelian scheme over $U$ and $\lambda$ is a principal polarization on $X$.

\subsubsection{} If $X\to S$ is a smooth morphism of schemes, the dual $\mathcal{O}_X$-module of the sheaf of relative differentials $\Omega^1_{X/S}$ (i.e. the sheaf of $\mathcal{O}_S$-derivations of $\mathcal{O}_X$) is denoted by $T_{X/S}$.  It is a vector bundle over $X$ whose rank is given by the relative dimension of $X \to S$. If $S=\Spec R$ is affine, we denote $T_{X/S}=T_{X/R}$.

 The \emph{Lie bracket} $[ \ , \ ] : T_{X/S} \times T_{X/S} \to T_{X/S}$ is defined on derivations by $[\theta_1,\theta_2] = \theta_1\circ \theta_2 - \theta_2\circ \theta_1$. 

If $S$ is a scheme, and $f: X \to Y$ is a morphism of smooth $S$-schemes, then there is a canonical morphism of $\mathcal{O}_X$-modules $f^*\Omega^1_{Y/S} \to \Omega^1_{X/S}$. Further, as $Y\to S$ is smooth, the canonical morphism of $\mathcal{O}_X$-modules  $f^*T_{Y/S} \to (f^*\Omega^1_{Y/S})^{\vee}$ is an isomorphism. We denote by
\begin{align*}
Df: T_{X/S} \to f^*T_{Y/S}
\end{align*}
the dual $\mathcal{O}_X$-morphism of $f^*\Omega^1_{Y/S} \to \Omega^1_{X/S}$ after the identification $(f^*\Omega^1_{Y/S})^{\vee}\cong f^*T_{Y/S}$. If $f$ is smooth, we have an exact sequence of vector bundles over $X$
\begin{align*}
0 \to T_{X/Y} \to T_{X/S} \stackrel{Df}{\to} f^*T_{Y/S} \to 0\text{.}
\end{align*}

\subsubsection{} If $U$ is any scheme, the category of $U$-schemes (resp. $U$-group schemes) is denoted by $\Sch_{/U}$ (resp. $\textsf{GpSch}_{/U}$). The category of sets is denoted by $\Sets$. If $\textsf{C}$ is any category, its opposite category is denoted by $\textsf{C}^{\opp}$.

\subsubsection{} \label{stacks}

We shall use the language of \emph{categories fibered in groupoids} and the elements of the theory of \emph{Deligne-Mumford stacks}. We follow the same conventions and terminology of \cite{olsson16}. In particular, if $S$ is a scheme, whenever we talk about a \emph{stack} over the category of $S$-schemes $\Sch_{/S}$ (cf. \cite{olsson16} Definition 4.6.1), or simply a stack over $S$ (or an $S$-stack), we shall always assume that $\Sch_{/S}$ is endowed with the \emph{étale topology}.

In view of \cite{olsson16} Corollary 8.3.5, by an \emph{algebraic space} over a scheme $S$ we mean a Deligne-Mumford stack $\mathcal{X}$ over $S$ such that for any $S$-scheme $U$ the fiber category $\mathcal{X}(U)$ is discrete (i.e. any automorphism is the identity).

The \emph{étale site} of a Deligne-Mumford stack $\mathcal{X}$ is denoted by $\text{Ét}(\mathcal{X})$ (cf. \cite{olsson16} Paragraph 9.1). We recall that the objects of the underlying category of $\Et(\mathcal{X})$ are \emph{étale schemes over $\mathcal{X}$}, that is, pairs $(U,u)$ where $U$ is an $S$-scheme and $u: U \to \mathcal{X}$ is an étale $S$-morphism; morphisms are given by couples $(f,f^b): (U',u')\to (U,u)$, where $f$ is an $S$-morphism and $f^b :u' \to u\circ f$ is an isomorphism of functors $U' \to \mathcal{X}$. Coverings in $\Et(\mathcal{X})$ are given by families of morphisms $\{(f_i,f_i^{b}):(U_i,u_i) \to (U,u)\}_{i\in I}$ such that $\{f_i:U_i \to U\}_{i\in I}$ is an étale covering of $U$.

The structural sheaf on $\Et(\mathcal{X})$, which to any $(U,u)$ associates the ring $\Gamma(U,\mathcal{O}_U)$, is denoted by $\mathcal{O}_{\mathcal{X}_{\et}}$. We recall that an $\mathcal{O}_{\mathcal{X}_{\et}}$-module $\mathcal{F}$ is said to be \emph{quasi-coherent} if $u^*\mathcal{F}$ is a quasi-coherent $\mathcal{O}_U$-module for any object $(U,u)$ of $\Et(\mathcal{X})$.

By a \emph{vector bundle} over a Deligne-Mumford stack $\mathcal{X}$, we mean a locally free $\mathcal{O}_{\mathcal{X}_{\et}}$-module of finite rank. We define subbundles, bases, and duals as in \ref{vectorbundles}.

\subsubsection{} \label{tangentstacks}  Sheaves of differentials and tangent sheaves can also be defined for Deligne-Mumford stacks. If $\mathcal{X}$ is a Deligne-Mumford stack over $S$, we define a presheaf of $\mathcal{O}_{\mathcal{X}_{\et}}$-modules $\Omega^1_{\mathcal{X}/S}$ on $\Et(\mathcal{X})$ by 
\begin{align*}
\Gamma((U,u),\Omega^1_{\mathcal{X}/S}) \defeq \Gamma(U,\Omega^1_{U/S})
\end{align*}
for any étale scheme $(U,u)$ over $\mathcal{X}$; restriction maps are defined in the obvious way. Since, for any étale morphism of $S$-schemes $f:U' \to U$, the induced morphism $f^*\Omega^1_{U/S} \to \Omega^1_{U'/S}$ is an isomorphism of $\mathcal{O}_{U'}$-modules, and for any $S$-scheme $U$ the sheaf $\Omega^1_{U/S}$ is a quasi-coherent $\mathcal{O}_U$-module, we see that $\Omega^1_{\mathcal{X}/S}$ is in fact a quasi-coherent sheaf over $\mathcal{X}$ (cf. \cite{olsson16} Lemma 4.3.3). Note that $u^*\Omega^1_{\mathcal{X}/S} = \Omega^1_{U/S}$ for any étale scheme $(U,u)$ over $\mathcal{X}$.

Let $\varphi: \mathcal{X} \to \mathcal{Y}$ be a morphism of Deligne-Mumford stacks over $S$. If $\varphi$ is representable by schemes, then there exists a unique morphism of $\mathcal{O}_{\mathcal{Y}}$-modules $\Omega^1_{\mathcal{Y}/S} \to \varphi_*\Omega^1_{\mathcal{X}/S}$ inducing, for any étale scheme $(V,v)$ over $\mathcal{Y}$, the canonical morphism $\Omega^1_{V/S} \to \varphi'_*\Omega^1_{U/S}$, where $(U,u)$ (resp. $\varphi' :U \to V$) denotes the étale scheme over $\mathcal{X}$ (resp. the morphism of $S$-schemes) obtained from $(V,v)$ (resp. $\varphi$) by base change. If, moreover, $\varphi$ is quasi-compact and quasi-separated, by adjointness (cf. \cite{olsson16} Proposition 9.3.6), we obtain a morphism of $\mathcal{O}_{\mathcal{X}_{\et}}$-modules
\begin{align}\label{pullbackmorphism}
\varphi^*\Omega^1_{\mathcal{Y}/S} \to \Omega^1_{\mathcal{X}/S}\text{.}
\end{align}
We then define a quasi-coherent $\mathcal{O}_{\mathcal{X}_{\et}}$-module
\begin{align*}
\Omega^1_{\mathcal{X}/\mathcal{Y}} \defeq \text{coker} (\varphi^*\Omega^1_{\mathcal{Y}/S} \to \Omega^1_{\mathcal{X}/S})\text{.}
\end{align*}

If $\mathcal{X}$ is a smooth Deligne-Mumford stack over $S$ then $\Omega^1_{\mathcal{X}/S}$ is a vector bundle over $\mathcal{X}$. We define $T_{\mathcal{X}/S}$ as the dual $\mathcal{O}_{\mathcal{X}_{\et}}$-module of $\Omega^1_{\mathcal{X}_/S}$. If $\varphi: \mathcal{X} \to \mathcal{Y}$ is a morphism of smooth Deligne-Mumford stacks over $S$ representable by smooth schemes, then  $\Omega^1_{\mathcal{X}/\mathcal{Y}}$ is a vector bundle over $\mathcal{X}$, and its dual is denoted by $T_{\mathcal{X}/\mathcal{Y}}$. Moreover, in this case, the morphism in (\ref{pullbackmorphism}) is injective and induces a surjective morphism of $\mathcal{O}_{\mathcal{X}_{\et}}$-modules $D\varphi : T_{\mathcal{X}/S} \to \varphi^*T_{\mathcal{Y}/S}$. We thus obtain an exact sequence of quasi-coherent $\mathcal{O}_{\mathcal{X}_{\et}}$-modules
\begin{align*}
0 \to T_{\mathcal{X}/\mathcal{Y}} \to T_{\mathcal{X}/S} \stackrel{D\varphi}{\to}\varphi^*T_{\mathcal{Y}/S} \to 0\text{.} 
\end{align*}

\section{Symplectic-Hodge bases} \label{shb}

We start this section by recalling the definition of the de Rham cohomology of an abelian scheme and its main properties. We next explain how to associate to a principal polarization on an abelian scheme a symplectic structure on its first de Rham cohomology. This leads us to the definition of symplectic-Hodge bases. 

\subsection{De Rham cohomology of abelian schemes}

Let $p:X \to U$ be an abelian scheme of relative dimension $g$.

 We recall that, for any integer $i\ge0$, the \emph{$i$-th de Rham cohomology} sheaf of $\mathcal{O}_U$-modules associated to $p$ is defined as the $i$-th left hyperderived functor of $p_*$ applied to the complex of relative differential forms $\Omega^{\bullet}_{X/U}$: 
\begin{align*}
H^i_{\dR}(X/U) \defeq \mathbf{R}^ip_*\Omega^{\bullet}_{X/U}\text{.}
\end{align*}
If $F:X \to Y$ is a morphism of abelian schemes over $U$, we denote by $F^* : H^i_{\dR}(Y/U) \to H^i_{\dR}(X/U)$ the induced $\mathcal{O}_U$-morphism on cohomology.

One can prove that there is a canonical isomorphism given by cup product
\begin{align*}
{\bigwedge}^i H^1_{\dR}(X/U) \stackrel{\sim}{\longrightarrow} H^i_{\dR}(X/U)\text{,} 
\end{align*}
and that $H^1_{\dR}(X/U)$ is a vector bundle over $U$ of rank $2g$. Moreover, the canonical $\mathcal{O}_U$-morphism $p_*\Omega^1_{X/U} \to H^1_{\dR}(X/U)$ induces an isomorphism of $p_*\Omega^1_{X/U}$ with a rank $g$ subbundle of $H^1_{\dR}(X/U)$, its \emph{Hodge subbundle} $F^1(X/U)$. It fits into a canonical exact sequence of $\mathcal{O}_U$-modules:
\begin{align} \label{hodgefiltr}
%\tag{H}
0 \longrightarrow F^1(X/U) \longrightarrow H^1_{\dR}(X/U) \longrightarrow R^1p_*\mathcal{O}_X \longrightarrow 0\text{.}
\end{align}
The formation of $H^1_{\dR}(X/U)$, $F^1(X/U)$, $R^1p_*\mathcal{O}_X$, and the above exact sequence is compatible with every base change in $U$.

 For a proof of all these facts, the reader may consult  \cite{BBM82} 2.5.

\subsection{Symplectic form associated to a principal polarization} \label{shb2}

 Let $p:X \to U$ be a projective abelian scheme of relative dimension $g$ and $\lambda :X \to X^t$ be a principal polarization.  In this paragraph, we recall how to associate to $\lambda$ a canonical \emph{symplectic} $\mathcal{O}_U$-bilinear form 
\begin{align*}
\langle \ , \ \rangle_{\lambda} : H^1_{\dR}(X/U) \times H^1_{\dR}(X/U) \longrightarrow \mathcal{O}_U\text{.}
\end{align*}
We refer to Appendix \ref{sympl} for basic definitions and terminology concerning symplectic forms on vector bundles over schemes.

Recall that to any line bundle $\mathcal{L}$ on $X$ we can associate its \emph{first Chern class in de Rham cohomology} $c_{1,\dR}(\mathcal{L})$, namely the global section of $H^2_{\dR}(X/U)$ given by the image of the class of the line bundle $\mathcal{L}$ under the morphism of $\mathcal{O}_U$-modules
\begin{align*}
R^1p_*\mathcal{O}_X^{\times} \to \mathbf{R}^1p_*\Omega^{\bullet}_{X/U}[1]\cong H^2_{\dR}(X/U)
\end{align*}
induced by $\dlog : \mathcal{O}_{X}^{\times} \to \Omega^{\bullet}_{X/U}[1]$.\footnote{We adopt the same sign conventions of \cite{BBM82} 0.3 for the differentials of the shifted complex $\Omega^{\bullet}_{X/U}[1]$ and for the isomorphism $\mathbf{R}^1p_*\Omega^{\bullet}_{X/U}[1]\cong H^2_{\dR}(X/U)$.}

We apply the above construction to the Poincaré line bundle $\mathcal{P}_{X/U}$ on the projective abelian scheme $X\times_U X^t$ over $U$. Let 
\begin{align*}
\phi_{X/U} : H^1_{\dR}(X/U)^{\vee} \longrightarrow H^1_{\dR}(X^t/U)
\end{align*}
be the morphism of  $\mathcal{O}_U$-modules given by the image of $c_{1,\dR}(\mathcal{P}_{X/U})$ in the Künneth component  $H^1_{\dR}(X/U) \tensor_{\mathcal{O}_U}H^1_{\dR}(X^t/U)$ of $H^2_{\dR}(X/U)$, and consider the  isomorphism of $\mathcal{O}_U$-modules 
\begin{align*}
\lambda^*: H^1_{\dR}(X^t/U) \to H^1_{\dR}(X/U)
\end{align*}
induced by the principal polarization $\lambda : X \to X^t$. For any sections $\gamma$ and $\delta$ of $H^1_{\dR}(X/U)^{\vee}$, we put
\begin{align*}
Q_{\lambda}(\gamma,\delta) \defeq \delta\circ\lambda^*\circ \phi_{X/U}(\gamma)\text{.}
\end{align*}
%alors Q_{\lambda}(\gamma, \ ) = \lambda^*\circ \phi_{X/U}(\gamma)

It is clear that $Q_{\lambda}$ defines an $\mathcal{O}_U$-bilinear form over $H^1_{\dR}(X/U)^{\vee}$. By \cite{BBM82} 5.1.3.1, $\phi_{X/U}$ is in fact an isomorphism; in particular, $Q_{\lambda}$ is non-degenerate. By duality, we can thus define a non-degenerate bilinear form $\langle \ , \ \rangle_{\lambda}$ over $H^1_{\dR}(X/U)$ via
\begin{align*}
\langle Q_{\lambda}( \gamma , \ ), Q_{\lambda}( \delta , \ )\rangle_{\lambda} \defeq Q_{\lambda}(\gamma,\delta)\text{,}
\end{align*}
where we identified $H^1_{\dR}(X/U)^{\vee \vee}$ with $H^1_{\dR}(X/U)$. 
%alors \langle \ , \beta \rangle_{\lambda} = (\lambda^*\circ \phi_{X/U})^{-1}\beta

\begin{lemma} \label{alternating}
The non-degenerate bilinear form $\langle \ , \ \rangle_{\lambda}$ is alternating, thus symplectic.
\end{lemma}

\begin{proof}
It suffices to prove that $Q_{\lambda}$ is alternating. Since $\lambda$ is a polarization, it is étale locally over $U$ induced by a line bundle $\mathcal{L}$ over $X$ relatively ample over $U$. We consider the first Chern class $c_{1,\dR}(\mathcal{L})$ in $H_{\dR}^2(X/U) \cong \bigwedge^2 H_{\dR}^1(X/U)$. Then, one can verify that $Q_{\lambda}$ defined above coincides with the alternating form 
\begin{align*}
(\gamma, \delta) \mapsto \gamma \wedge \delta ( c_{1,\dR}(\mathcal{L}))\text{.}
\end{align*}
We refer to \cite{DP94}, Section 1, for further details. 
\end{proof}

Thus we obtain a symplectic vector bundle $(H^1_{\dR}(X/U),\langle \ , \ \rangle_{\lambda})$ over $U$ in the sense of Definition \ref{defisymplbundle}.

\begin{lemma} \label{f1lagrangian}
$F^1(X/U)$ is a Lagrangian subbundle of $H^1_{\dR}(X/U)$ with respect to the symplectic form $\langle \ , \ \rangle_{\lambda}$.
\end{lemma}

\begin{proof}
Since the rank of $H^1_{\dR}(X/U)$ is $2g$, and $F^1(X/U)$ is a rank $g$ subbundle of $H^1_{\dR}(X/U)$, it suffices to prove that $F^1(X/U)$ is isotropic with respect to $\langle \ , \ \rangle_{\lambda}$ (cf. Corollary \ref{corosympl}). This follows immediately from the compatibility of $\phi_{X/U}$ with the exact sequence (\ref{hodgefiltr}), that is, from the existence of canonical morphisms $\phi_{X/U}^0$ and $\phi_{X/U}^1$ making the diagram
$$
  \raisebox{-0.5\height}{\includegraphics{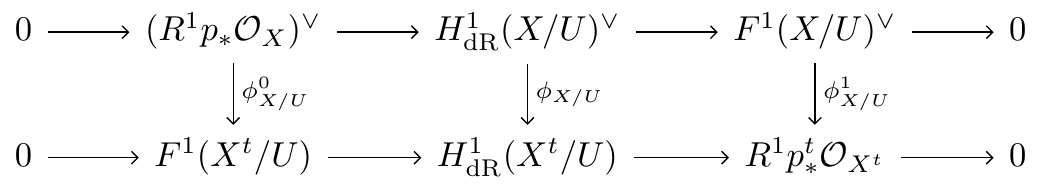}}
$$
% $$
% \begin{tikzcd}
% 0 \arrow{r} & (R^1p_*\mathcal{O}_X)^{\vee} \arrow{r} \arrow{d}{\phi_{X/U}^0} & H^1_{\dR}(X/U)^{\vee} \arrow{r} \arrow{d}{\phi_{X/U}} & F^1(X/U)^{\vee} \arrow{r} \arrow{d}{\phi_{X/U}^1} & 0 \\
% 0 \arrow{r} & F^1(X^t/U) \arrow{r} & H^1_{\dR}(X^t/U) \arrow{r} & R^1p^t_*\mathcal{O}_{X^t} \arrow{r} & 0   
% \end{tikzcd}
% $$
commute (\cite{BBM82} Lemme 5.1.4; the morphisms $\phi_{X/U}^0$ and $\phi_{X/U}^1$ are uniquely determined by this commutative diagram, and are isomorphisms).
\end{proof}

\begin{obs} \label{remarkbasechange}
It is clear from the above construction that the formation of the symplectic form $\langle \ , \ \rangle_{\lambda}$ is compatible with base change. Namely, if $f:U' \to U$ is a morphism of schemes, and $(X',\lambda')$ denotes the principally polarized abelian scheme over $U'$ obtained by base change via $f$, then $f^*\langle \ , \ \rangle_{\lambda}$ coincides with $\langle \ , \ \rangle_{\lambda'}$ under the base change isomorphism $f^*H^1_{\dR}(X/U) \stackrel{\sim}{\to} H^1_{\dR}(X'/U')$. 
\end{obs}

\subsection{Symplectic-Hodge bases of $H^1_{\dR}(X/U)$} 

Let $U$ be a scheme and $(X,\lambda)$ be a principally polarized abelian scheme over $U$ of relative dimension $g$.

\begin{defi}
 A \emph{symplectic-Hodge basis} of $(X,\lambda)_{/U}$ is a sequence $b=(\omega_1,\ldots,\omega_g,\eta_1,\ldots,\eta_g)$ of $2g$ global sections of $H^1_{\dR}(X/U)$ such that:
\begin{enumerate}
   \item $\omega_1,\ldots,\omega_g$ are sections of $F^1(X/U)$, and
   \item $b$ is a symplectic basis of $(H^1_{\dR}(X/U),\langle \ , \ \rangle_{\lambda})$ (Definition \ref{defsymplbasis}).
\end{enumerate} 
\end{defi}

Let us note that symplectic-Hodge bases may not exist globally, but such bases always exist locally for the Zariski topology over $U$ by Proposition \ref{exisunic}.

\section{The moduli stack $\mathcal{B}_g$} \label{modulistack}

In this section, we define for every integer $g\ge 1$ a category $\mathcal{B}_g$ fibered in groupoids over the category of schemes $\Sch_{/\ZZ}$ classifying principally polarized abelian schemes of relative dimension $g$ endowed with a symplectic-Hodge basis. 

We prove that $\mathcal{B}_g \to \Spec \ZZ$ is a smooth Deligne-Mumford stack over $\Spec \ZZ$ of relative dimension $2g^2+g$. The main point in proving this result will be to remark that for any principally polarized abelian scheme $(X,\lambda)$ of relative dimension $g$ over an affine scheme $U=\Spec R$, there is a natural free and transitive right action of the Siegel parabolic subgroup $P_g(R)$ of $\Sp_{2g}(R)$, consisting of upper triangular matrices, on the set of symplectic-Hodge bases of $(X,\lambda)_{/U}$.   

\subsection{The moduli stack $\mathcal{A}_g$} 

Let $g\ge 1$ be an integer. To fix ideas and notations we recall the definition of the moduli stack of principally polarized abelian schemes of relative dimension $g$. 

For any scheme $S$, we define a category fibered in groupoids $\mathcal{A}_{g,S} \to \Sch_{/S}$ as follows.
\begin{enumerate}[(i)]
    \item An object of $\mathcal{A}_{g,S}$ is given by an $S$-scheme $U$ and a principally polarized abelian scheme $(X,\lambda)$ of relative dimension $g$ over $U$; when $U$ is not clear in the context, we shall incorporate it in the notation by writing $(X,\lambda)_{/U}$. A morphism $(X,\lambda)_{/U} \to (Y,\mu)_{/V}$ in $\mathcal{A}_{g,S}$, denoted $F_{/f}$, is given by a cartesian diagram of $S$-schemes
$$
  \raisebox{-0.5\height}{\includegraphics{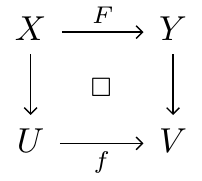}}
$$
%  $$
% \begin{tikzcd}
%   X \arrow{r}{F}\arrow{d} & Y \arrow{d}\\
%   U \arrow{r}[swap]{f} & V \arrow[draw=none]{ul}[description]{\square}
% \end{tikzcd}
% $$
preserving the identity sections of the abelian schemes and identifying $\lambda$ with the pullback of $\mu$ by $f:U \to V$. We shall occasionally denote $F_{/f}$ simply by $F$ when there will be no danger of confusion. We may also denote $(X,\lambda) = (Y,\mu)\times_UV$. 
 \item The structural functor $\mathcal{A}_{g,S} \to \Sch_{/S}$ is given by sending an object $(X,\lambda)_{/U}$ of $\mathcal{A}_{g,S}$ to the $S$-scheme $U$, and a morphism $F_{/f}$ to $f$.
\end{enumerate}
 
If $S=\Spec R$ is affine (resp. $S=\Spec \ZZ$), then we denote $\mathcal{A}_{g,S}\eqdef\mathcal{A}_{g,R}$ (resp. $\mathcal{A}_{g,S}\eqdef\mathcal{A}_g$).

Recall that the category of $S$-schemes can be seen as a subcategory of the 2-category of categories fibered in groupoids over $\Sch_{/S}$ by sending each $S$-scheme $U$ to the category $\Sch_{/U}$ endowed with its natural functor $\Sch_{/U} \to \Sch_{/S}$. In the sequel, we shall adopt the standard convention of denoting $\Sch_{/U}$ simply by $U$ when working in the context of categories fibered in groupoids. Then $\mathcal{A}_{g,S}$ is canonically equivalent to $\mathcal{A}_g\times_{\ZZ}S$ as categories fibered in groupoids over $S$.

We summarize the main properties of $\mathcal{A}_{g,S}$ we are going to use in the form of the next theorem.

\begin{theorem} \label{DMstack}
For any scheme $S$ and any integer $g\ge 1$, $\mathcal{A}_{g,S}$ is a smooth Deligne-Mumford stack over $S$ of relative dimension $g(g+1)/2$. %Moreover, $\mathcal{A}_g$ admits a quasi-projective coarse moduli space $A_g$ over $\ZZ$.
\end{theorem}

A proof that $\mathcal{A}_{g,S}$ is a Deligne-Mumford stack over $S$ is essentially contained in \cite{GIT94} Theorem 7.9 (cf. \cite{olsson12} Theorem 2.1.11). Smoothness and relative dimension are obtained by a theorem of Grothendieck (cf. \cite{oort71} Proposition 2.4.1). %By a theorem of Keel and Mori (cf. \cite{olsson16} Theorem 11.1.2), $\mathcal{A}_g$ admits a coarse moduli space $A_g$ over $\ZZ$ in the category of algebraic spaces; it follows from \cite{moret-bailly85}, VII, Théorème 4.2, that $A_g$ is actually a quasi-projective scheme over $\ZZ$.

\subsection{Definition of $\mathcal{B}_g$} \label{defbg}

Let $F_{/f}: (X,\lambda)_{/U} \to (Y,\mu)_{/V}$ be a morphism in $\mathcal{A}_{g}$. By the compatibility with base change of the symplectic forms induced by principal polarizations (Remark \ref{remarkbasechange}), the pullback $F^*b$ of every symplectic-Hodge basis $b$ of $(Y,\mu)_{/V}$ is a symplectic-Hodge basis of $(X,\lambda)_{/U}$. We can thus define a functor
\begin{align*}
 \underline{B}_g : \mathcal{A}_{g}^{\opp} \longrightarrow \Sets
\end{align*}
that sends every object $(X,\lambda)_{/U}$ of $\mathcal{A}_{g}$ to the set of symplectic-Hodge bases of $(X,\lambda)_{/U}$, and whose action on morphisms is given by pullbacks as above. 

From the functor $\underline{B}_g$, we form a category fibered in groupoids 
\begin{align*}
\mathcal{B}_g \to \Spec \ZZ
\end{align*}
 as follows.
\begin{enumerate}[(i)]
      \item  An object of $\mathcal{B}_g$ is a ``triple'' $(X,\lambda,b)_{/U}$ where $(X,\lambda)_{/U}$ is an object of $\mathcal{A}_{g}$ and $b \in \underline{B}_g(X,\lambda)$. An arrow $(X,\lambda,b)_{/U} \to (Y,\mu,c)_{/V}$ is given by a morphism $F_{/f}: (X,\lambda)_{/U} \to (Y,\mu)_{/V}$ such that $b=F^*c$. We denote by
\begin{align*}
\pi_g: \mathcal{B}_g \to \mathcal{A}_{g}
\end{align*}
     the forgetful functor $(X,\lambda,b)_{/U}\mapsto (X,\lambda)_{/U}$.
     \item The structural functor $\mathcal{B}_g \to \Spec \ZZ$ is defined as the composition of $\pi_g$ with the structural functor $\mathcal{A}_{g} \to \Spec \ZZ$.
\end{enumerate}

The rest of this section is devoted to the proof of the next theorem.

\begin{theorem} \label{smoothdmstack}
The category fibered in groupoids $\mathcal{B}_g\to \Spec \ZZ$ is a smooth Deligne-Mumford stack over $\Spec \ZZ$ of relative dimension $2g^2+g$.
\end{theorem}

\subsection{Siegel parabolic subgroup and proof of Theorem \ref{smoothdmstack}} \label{uppertriang}

Fix a scheme $U$ and an object $(X,\lambda)$ of $\mathcal{A}_{g}$ lying over $U$. Then we can define a functor
\begin{align*}
\underline{B}_{(X,\lambda)} : \Sch_{/U}^{\opp} \longrightarrow \Sets
\end{align*}  
that sends a $U$-scheme $U'$ to the set $\underline{B}_g((X,\lambda)\times_UU')$. It is clear that this functor defines a sheaf for the Zariski topology over $\Sch_{/U}$.

Let us now consider the \emph{symplectic group} $\Sp_{2g}$, namely the smooth affine group scheme over $\Spec \ZZ$ of relative dimension $2g^2+g$ such that for every affine scheme $V=\Spec R$
\begin{align*}
{\Sp}_{2g}(V) = \left.\left\{\left(\begin{array}{cc}
                           A & B \\
                           C & D
                          \end{array} \right) \in M_{2g\times 2g}(R) \ \right| \begin{aligned} &\ \ \ \ \ \ \ \ \ \ \ \   A,B,C,D \in M_{g\times g}(R) \text{ satisfy  } \\ &AB\transp = BA\transp\text{, }  CD\transp= DC\transp\text{, and } AD\transp -BC\transp = \mathbf{1}_g \end{aligned}\right\}\text{.}
 \end{align*}
% \begin{align*}
% {\Sp}_{2g}(V)  =  \{M \in M_{2g\times 2g}(R) \mid MJM\transp =J\}
% \end{align*}
% where
% \begin{align*}
% J\defeq \left(\begin{array}{cc}
%                   0 & \mathbf{1}_g \\
%                   -\mathbf{1}_g& 0
%                   \end{array}\right)\text{.}
% \end{align*}

% \begin{obs} \label{eqsp}
% As $J^2=-\mathbf{1}_{2g}$, the condition $MJM\transp = J$ is equivalent to $M^{-1} = -JM\transp J$; thus $MJM\transp = J$ if and only if $M\transp J M = J$. In particular, if we write
% \begin{align*}
% M = \left(\begin{array}{cc}
%                   A & B \\
%                   C & D
%                   \end{array}\right) \in M_{2g\times 2g}(R)
% \end{align*}
% for some $A,B,C,D\in M_{g\times g}(R)$, then $M$ is in $\Sp_{2g}(R)$ if and only if one of the following two conditions is satisfied
% \begin{enumerate}
%    \item $AB\transp = BA\transp$, $CD\transp= DC\transp$, and $AD\transp -BC\transp = \mathbf{1}_g$.
%    \item $A\transp C = C\transp A\text{, } B\transp D = D\transp B\text{, and } A\transp D - C\transp B = \mathbf{1}_g$.
% \end{enumerate}
% \end{obs}

% Let us now consider the symplectic group $\Sp_{2g}$, namely the smooth affine group scheme over $\Spec \ZZ$ such that for every affine scheme $V=\Spec R$

% \begin{obs}
% \label{conditions}
% For later reference, we recall that the above equations defining the symplectic group are equivalent to $A\transp C = C\transp A\text{, } B\transp D = D\transp B\text{, and } A\transp D - C\transp B = 1_g$.
% \end{obs}

The \emph{Siegel parabolic subgroup} $P_g$ of $\Sp_{2g}$ is defined as the subgroup scheme of $\Sp_{2g}$ such that, for every affine scheme $V=\Spec R$,
\begin{align*} P_g(V) = 
 \left.\left\{\left(\begin{array}{cc}
                           A & B \\
                           0 & (A\transp)^{-1}
                          \end{array} \right) \in M_{2g\times 2g}(R) \ \right|\  A \in {\GL}_g(R)\text{ and } B\in M_{g\times g}(R)\text{ satisfy }AB\transp=BA\transp \right\}\text{.}
\end{align*}
Note that $P_g$ is a smooth affine group scheme over $\Spec \ZZ$ of relative dimension $g(3g+1)/2$.

Let $(X,\lambda,b)$ be an object of $\mathcal{B}_g$ lying over $V=\Spec R$ and consider $b = (\  \omega \ \  \eta\  )$ as a row vector of order $2g$ with coefficients in the $R$-module $H^1_{\dR}(X/V)$. For any
\begin{align*}
p=\left(\begin{array}{cc}
                           A & B \\
                           0 & (A\transp)^{-1}
                          \end{array} \right) \in P_g(V)
\end{align*}
it easy to check that
\begin{align*}
b\cdot p \defeq (\begin{array}{cc}\omega A &  \omega B + \eta (A\transp)^{-1}\end{array})
\end{align*}
is a symplectic-Hodge basis of $(X,\lambda)_{/V}$. This defines a right action of $P_g(V)$ on $\underline{B}_g(X,\lambda)$:
\begin{align*}
\underline{B}_g(X,\lambda) \times P_g(V) \longrightarrow \underline{B}_g(X,\lambda)\text{.}
\end{align*}
 Moreover, it is clear that if $V'\subset V$ is an affine open subscheme of $V$, then the natural diagram
$$
\raisebox{-0.5\height}{\includegraphics{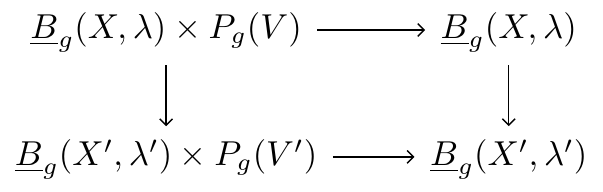}}
$$
% $$
% \begin{tikzcd}
% \underline{B}_g(X,\lambda) \times P_g(V) \arrow{r}\arrow{d}& \underline{B}_g(X,\lambda)\arrow{d}\\
% \underline{B}_g(X',\lambda') \times P_g(V') \arrow{r}& \underline{B}_g(X',\lambda')
% \end{tikzcd}
% $$
commutes, where $(X',\lambda')=(X,\lambda)\times_VV'$.

Thus, for any scheme $U$, and any object $(X,\lambda)$ of $\mathcal{A}_{g}$ lying over $U$, we obtain a right action of the $U$-group scheme $P_{g,U} = P_g\times_{\ZZ}U$ on $\underline{B}_{(X,\lambda)}$.

\begin{lemma} \label{torsor}
The Zariski sheaf $\underline{B}_{(X,\lambda)}$ over $\mathsf{Sch}_{/U}$ is a right Zariski $P_{g,U}$-torsor for the above action.
\end{lemma}

\begin{proof}
If $V$ is any affine scheme over $U$ such that $\underline{B}_{(X,\lambda)}(V)$ is non-empty, a routine computation shows that the action of $P_g(V)$ on $\underline{B}_{(X,\lambda)}(V)$ is free and transitive. Moreover, it was already remarked above that symplectic-Hodge bases exist locally for the Zariski topology.
\end{proof}

Since $P_{g,U}$ is affine, smooth, and of relative dimension $g(3g+1)/2$ over $U$, Lemma \ref{torsor} immediately implies the following corollary.

\begin{coro} \label{relrepr0}
For every scheme $U$, and every object $(X,\lambda)$ of $\mathcal{A}_{g}$ lying over $U$, the functor $\underline{B}_{(X,\lambda)}$ is representable by a smooth affine $U$-scheme $B(X,\lambda)$ of relative dimension $g(3g+1)/2$.
\end{coro}

\begin{obs} \label{relrepr}
Let us keep the notations of the above corollary. Recall that the principally polarized abelian scheme $(X,\lambda)$ over $U$ corresponds to a morphism $U \to \mathcal{A}_g$. Then $B(X,\lambda)$ is a scheme representing $\mathcal{B}_g\times_{\mathcal{A}_g}U$.
\end{obs} 

\begin{proof}[Proof of Theorem \ref{smoothdmstack}]
Recall that for any scheme $U$ and any abelian scheme $X$ over $U$, $H^1_{\dR}(X/U)$ is a quasi-coherent sheaf over $U$, and that any quasi-coherent sheaf over $U$ induces a sheaf over $\Sch_{/U}$ endowed with the fppf topology (\cite{olsson16} Lemma 4.3.3). Since the étale topology is coarser than the fppf topology, this shows in particular that $H^1_{\dR}(X/U)$ induces a sheaf over $\Sch_{/U}$ endowed with the étale topology; this immediately implies that $\mathcal{B}_g\to \Spec \ZZ$ is a stack over $\Spec \ZZ$. 

It follows in particular from Corollary \ref{relrepr0} that the morphism $\pi_g : \mathcal{B}_g \to \mathcal{A}_{g}$ is representable by smooth schemes (Remark \ref{relrepr}). Hence, as $\mathcal{A}_{g}\to \Spec \ZZ$ is a Deligne-Mumford stack over $\Spec \ZZ$, the same holds for $\mathcal{B}_g\to \Spec \ZZ$ (cf. \cite{olsson16} Proposition 10.2.2). The smoothness of $\mathcal{B}_g \to \Spec \ZZ$ follows by composition from that of $\mathcal{A}_{g} \to \Spec \ZZ$ and that of $\pi_g$. Finally, we can compute the relative dimension of $\mathcal{B}_g \to \Spec \ZZ$ as the sum of that of  $\mathcal{A}_{g} \to \Spec \ZZ$ and that of $\pi_g$:
\begin{align*}
\frac{g(g+1)}{2} + \frac{g(3g +1)}{2} = 2g^2+g\text{.}
\end{align*}  
\end{proof}

\section{Representability of $\mathcal{B}_g$ by a scheme} \label{representability}

It is easy to see that if $S$ is a scheme over $\mathbf{F}_2$, then $\mathcal{B}_g\times_{\ZZ} S\to S$ is not representable. Indeed, if $(X,\lambda,b)_{/U}$ is an object of $\mathcal{B}_g$ lying over a scheme $U$ over $\mathbf{F}_2$, then the involution $[-1]: P \mapsto -P$ on $X$ defines a non-trivial automorphism $[-1]_{/\id_U}: (X,\lambda)_{/U} \to (X,\lambda)_{/U}$ in $\mathcal{A}_g(U)$ such that
\begin{align*}
[-1]^*b = -b = b\text{,}
\end{align*}
thus a non-trivial automorphism of $(X,\lambda,b)_{/U}$ in $\mathcal{B}_g(U)$.

For any ring $R$, let us denote $\mathcal{B}_{g,R} \defeq \mathcal{B}_g\tensor_{\ZZ}R$. In this section we prove the following theorem.

\begin{theorem} \label{repr}
The stack $\mathcal{B}_{g,\ZZ[1/2]}\to \Spec \ZZ[1/2]$ is representable by a smooth quasi-projective scheme $B_g$ over $\ZZ[1/2]$ of relative dimension $2g^2+g$.
\end{theorem}

% In the sequel, we shall denote the universal principally polarized abelian scheme over $B_g$ by  $(X_g,\lambda_g)$, and the universal symplectic-Hodge basis of $(X_g,\lambda_g)_{/B_g}$ by $b_g$.

Let us briefly summarize our proof of Theorem \ref{repr}.

 We shall first prove that $\mathcal{B}_{g,\ZZ[1/2]}$ is an algebraic space over $\ZZ[1/2]$. This amounts to proving that the functor $\underline{B}_g$ is \emph{rigid} over $\ZZ[1/2]$ (see Definition \ref{defrigid} below). By the classical ``rigidity lemma'' for abelian schemes (Lemma \ref{rigidite}), we reduce the proof that $\underline{B}_g$ is rigid over $\ZZ[1/2]$ to proving that $\underline{B}_g$ is rigid over any algebraically closed field of characteristic $0$ or $p>2$. In positive characteristic, this will be obtained by a theorem of Oda characterizing the first de Rham cohomology of an abelian variety over a perfect field of characteristic $p$ in terms of its $p$-torsion subgroup scheme.

Finally, we use the existence of a quasi-projective surjective étale scheme over $\mathcal{A}_{g,\ZZ[1/2]}$ to conclude, via a simple base-change argument, that $\mathcal{B}_{g,\ZZ[1/2]}$ is actually representable by a quasi-projective $\ZZ[1/2]$-scheme.

\subsection{Rigidity of $\underline{B}_g$}

Let $R$ be a ring. The following terminology has been borrowed from \cite{KM85} 4.4.

\begin{defi} \label{defrigid}
We say that the functor $\underline{B}_g$ (cf. paragraph \ref{defbg}) is \emph{rigid} over $R$ if, for every $R$-scheme $U$, and every object $(X,\lambda)$ of $\mathcal{A}_g$ lying over $U$, the action of $\Aut_U(X,\lambda)$ on $\underline{B}_g(X,\lambda)_{/U}$ is free. 
\end{defi}

Note that $\underline{B}_g$ is rigid over $R$ if and only if the fiber categories of $\mathcal{B}_{g,R} \to \Spec R$ are discrete. As $\mathcal{B}_g$ is a Deligne-Mumford stack over $\Spec \ZZ$, this amounts to saying that $\mathcal{B}_{g,R} \to \Spec R$ is an algebraic space over $\Spec R$ (see our terminology conventions in \ref{stacks}).

\begin{lemma} \label{rig0}
Let $k$ be a field of characteristic 0. Then $\underline{B}_g$ is rigid over $k$.
\end{lemma}

\begin{proof}
Let $(X,\lambda,b)$ be an object of $\mathcal{B}_g$ lying over $k$ and $\varphi :X \to X$ be a $k$-automorphism of $(X,\lambda)$ such that $\varphi^*b = b$; we must show that $\varphi = \id_X$.

We claim that it is sufficient to treat the case $k=\CC$. In fact, as $X$ is of finite type over $k$, by ``elimination of noetherian hypothesis'' (cf. \cite{EGAIV3} 8.8, 8.9, 8.10, 12.2.1, and \cite{EGAIV4} 17.7.9), there exists a subfield $k_0$ of $k$, of finite type over $\QQ$, and a principally polarized abelian variety $(X_0,\lambda_0)$ over $k_0$ endowed with a symplectic-Hodge basis $b_0$ and  a $k_0$-automorphism $\varphi_0$ of $(X_0,\lambda_0)$ satisfying $\varphi_0^*b_0=b_0$, such that $(X,\lambda,b)$ (resp. $\varphi$) is obtained from $(X_0,\lambda_0,b_0)$ (resp. $\varphi_0$) by the base change $\Spec k \to \Spec k_0$. After fixing an embedding of $k_0$ in $\CC$, we finally remark that if $\varphi_{0,\CC}$ is the identity over $X_0 \tensor_{k_0}\CC$, then the same holds for $\varphi_0$, and thus also for $\varphi$.

Let then $k=\CC$. It is sufficient to prove that the induced automorphism of complex Lie groups $\varphi^{\an} : X^{\an} \to X^{\an}$ is the identity. As $X^{\an}$ is a complex torus, the exponential $\exp : \Lie X \to X^{\an}$ is a surjective morphism of complex Lie groups. Therefore, it follows from the commutative diagram
$$
\raisebox{-0.5\height}{\includegraphics{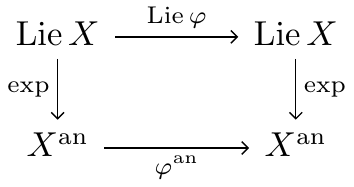}}
$$
% $$
% \begin{tikzcd}[column sep=large]
% \Lie X \arrow{r}{\Lie \varphi} \arrow{d}[swap]{\exp} & \Lie X\arrow{d}{\exp}\\
% X^{\an} \arrow{r}[swap]{\varphi^{\an}} & X^{\an}
% \end{tikzcd}
% $$
that it sufficient to prove that $\Lie \varphi = \id_{\Lie X}$. Now, if $\varphi$ preserves symplectic-Hodge basis of $(X,\lambda)$, then in particular the $\CC$-linear map $\varphi^* : H^0(X,\Omega^1_{X/\CC}) \to H^0(X,\Omega^1_{X/\CC})$ is the identity, and thus its dual $\Lie \varphi : \Lie X \to \Lie X$ is also the identity. 
\end{proof}

We now treat the case of positive characteristic. Let us briefly recall some notions in Dieudonné theory and its relations with abelian varieties.

Let $k$ be a perfect field of characteristic $p>0$. We denote by $W$ the ring of Witt vectors over $k$, and by $\sigma$ the unique ring automorphism of $W$ lifting the absolute Frobenius $x \mapsto x^p$ of $k$. We can then define a $W$-algebra $D$ generated by elements $F$ and $V$ subjected to the relations
\begin{align*}
FV = VF = p\text{, }\ \ \ Fx = \sigma(x)F\text{,  }\ \ \ 
xV=V\sigma(x)
\end{align*} 
for any $x\in W$.

The theory of Dieudonné (cf. \cite{oda69} Definition 3.12) provides an additive contravariant functor
\begin{align} \label{functordieudonne}
G \mapsto M(G)
\end{align} 
from the category of commutative finite $k$-group schemes of $p$-power order to the category of left $D$-modules. This functor is shown to be faithful and its essential image is given by the category of left $D$-modules of finite $W$-length: $M(G)$ is of $W$-length $r$ if and only if $G$ is of order $p^r$ (\cite{oda69} Corollary 3.16).

Now, let $X$ be an abelian variety over $k$ and consider the $k$-vector space $H^1_{\dR}(X/k)$ as a $W$-module via the canonical map $W \to k$. Then one can endow $H^1_{\dR}(X/k)$ with the structure of a $D$-module, the action of $F$ (resp. $V$) being induced by the relative Frobenius on $X$ (resp. the Cartier operator in degree 1); we refer to \cite{oda69} Definition 5.3 and Definition 5.6 for further details. This construction is functorial in the sense that for any morphism $\varphi:X \to Y$ of abelian varieties over $k$, if we endow $H^1_{\dR}(X/k)$ and $H^1_{\dR}(Y/k)$ with the preceding $D$-module structure, then the induced morphism on de Rham cohomology $\varphi^*: H^1_{\dR}(Y/k) \to H^1_{\dR}(X/k)$ is $D$-linear.

In the next statement, for any abelian variety $X$ over $k$, we regard $H^1_{\dR}(X/k)$ with the above $D$-module structure, and we denote its $p$-torsion subscheme by $X[p]$. Note that $X[p]$ is a commutative finite $k$-group scheme of order $p^{2\dim X}$.

\begin{theorem}[Oda, \cite{oda69} Corollary 5.11] \label{odathm}
The contravariant functors $X \mapsto M(X[p])$ and $X\mapsto H^1_{\dR}(X/k)$ from the category of abelian varieties over $k$ to the category of ($p$-torsion) $D$-modules of finite $W$-length are naturally equivalent.
\end{theorem} 

\begin{lemma} \label{rigp}
Let $k$ be a perfect field of characteristic $p>2$. Then $\underline{B}_g$ is rigid over $k$.
\end{lemma}

\begin{proof}
Let $(X,\lambda)$ be a principally polarized abelian variety over $k$ of dimension $g$ and $\varphi :X \to X$ be a $k$-automorphism of $(X,\lambda)$.

If $\varphi$ preserves a symplectic-Hodge basis of $(X,\lambda)_{/k}$, then in particular $\varphi^* : H^1_{\dR}(X/k) \to H^1_{\dR}(X/k)$ is the identity; a fortiori, $\varphi$ induces the identity on $H^1_{\dR}(X/k)$ regarded as a $D$-module. Then, by Theorem \ref{odathm}, $\varphi$ induces the identity on the $D$-module $M(X[p])$. As the functor $G\mapsto M(G)$ in (\ref{functordieudonne}) is faithful, $\varphi$ restricts to the identity on the $p$-torsion subscheme $X[p]$ of $X$. As $\varphi$ preserves, in addition, the polarization $\lambda$ on $X$, and since $p\ge 3$, then necessarily $\varphi=\id_X$ (cf. \cite{mumford70} IV.21, Theorem 5).   
\end{proof}

Recall the following version of the classical ``rigidity lemma'' for abelian schemes which follows from the arguments in the proof of Proposition 6.1 in \cite{GIT94}.

\begin{lemma} \label{rigidite}
Let $A$ be a local Artinian ring, and $X$ be an abelian scheme over $A$. If $\varphi:X \to X$ is an endomorphism of $A$-group schemes restricting to the identity on the closed fiber of $X \to \Spec A$, then $\varphi=\id_X$.
\end{lemma}

% \begin{proof}
% Let $p:X\to \Spec A$ denote the structural morphism and $e: \Spec A \to X$ the identity section. If $\varphi$ is as in the statement, then $f \defeq \varphi - \id$ restricts to the ``zero morphism'' on the closed fiber. As $\Spec A$ consisting in only one point, we obtain the set-theoretical identity $f=e \circ p$. To finish the proof, we must show that $f=e\circ p$ as morphisms of schemes.

% Let us momentarily denote the set-theoretical identity section $\Spec A \to X$ by $s$. Now, since $p$ is proper, smooth and with geometrically connected fibers, the induced morphism of sheaves $p^{\#}:\mathcal{O}_{\Spec A} \to p_*\mathcal{O}_X$ is an isomorphism, and we conclude that we can endow $s$ with a unique morphism of locally ringed spaces $s^{\#} : \mathcal{O}_X \to s_* \mathcal{O}_{\Spec A}$ such that $f=s\circ p$ as morphisms of schemes. Finally, we remark that
% \begin{align*}
% s = s \circ p \circ e = f\circ e = e
% \end{align*}  
% as morphisms of schemes.
% \end{proof}

\begin{prop} \label{rigid}
The functor $\underline{B}_g$ is rigid over $\ZZ[1/2]$.
\end{prop}

\begin{proof}
Let $U$ be a $\ZZ[1/2]$-scheme, $(X,\lambda)$ be an object of $\mathcal{A}_{g}$ lying over $U$, and $\varphi$ be an automorphism of $(X,\lambda)$ in the fiber category $\mathcal{A}_g(U)$ preserving an element $b$ of $\underline{B}_g(X,\lambda)$. We must show that $\varphi=\id_X$. This being a local property over $U$, we can assume that $U$ is affine.

Suppose that $U$ is noetherian. By Lemmas \ref{rig0} and \ref{rigp}, for every geometric point $\overline{u}$ of $U$, we have $\varphi_{X_{\overline{u}}}=\id_{X_{\overline{u}}}$. Let $Z$ be the closed subscheme of $U$ where $\varphi = \id$. Then $Z$ contains every closed point of $U$. By  Lemma \ref{rigidite}, and Krull's intersection theorem, $Z$ is also an open subscheme of $U$; hence $Z=U$, which amounts to saying that $\varphi = \id_X$.

In general, by ``elimination of noetherian hypothesis'' (cf. \cite{EGAIV3}, 8.8, 8.9, 8.10, 12.2.1, and \cite{EGAIV4}, 17.7.9), there exists an affine noetherian scheme $U_0$ under $U$, and a principally polarized abelian scheme $(X_0,\lambda_0)$ over $U_0$ endowed with a symplectic-Hodge basis $b_0$, and with an $U_0$-automorphism $\varphi_0$, such that $\varphi_0^*b_0=b_0$, and $(X,\lambda)$ (resp. $b$, resp. $\varphi$) is deduced from $(X_0,\lambda_0)$ (resp. $b_0$, resp. $\varphi_0$) by the base change $U\to U_0$. The preceding paragraph shows that $\varphi_0=\id_{X_0}$, hence $\varphi = \id_X$.
\end{proof}

\subsection{Proof of Theorem \ref{repr}}

We briefly recollect some facts on quotients of schemes by actions of finite groups. 

Let $S$ be a scheme and $\Gamma$ be a finite constant group scheme over $S$, that is, an $S$-group scheme associated to a finite abstract group $|\Gamma|$. 

For any $S$-scheme $X$, an $S$-action of $\Gamma$ on $X$ is equivalent to a morphism of groups $|\Gamma| \to \Aut_S(X)$. If $X$ is an $S$-scheme, we say that an action of $\Gamma$ on $X$ is \emph{free} if the action of $\Gamma(U)$ on $X(U)$ is free for any $S$-scheme $U$.

The next lemma easily follows from \cite{SGA103} V and \cite{knutson71} IV.1.

\begin{lemma} \label{finiteaction}
Let $S$ be an affine noetherian scheme and $X$ be a quasi-projective $S$-scheme equipped with an $S$-action of a finite constant group scheme $\Gamma$ over $S$. Then
\begin{enumerate}
     \item There exists a quasi-projective $S$-scheme $Y$ and a $\Gamma$-invariant surjective finite morphism $p:X \to Y$ such that the natural morphism of sheaves of rings over $Y$
\begin{align*}
\mathcal{O}_Y\to (p_*\mathcal{O}_X)^{|\Gamma|}
\end{align*}
is an isomorphism. We denote $Y\eqdef X/\Gamma$.%In particular, $p$ is a \emph{quotient} of $X$ by $\Gamma$ in the category of $S$-schemes, that is, for any $S$-scheme $Z$, every $\Gamma$-invariant morphism $X \to Z$ factors through $p$.
     %\item The morphism $p$ is finite and surjective. 
     % \item $Y$ is quasi-projective over $S$.
     \item If moreover the action of $\Gamma$ on $X$ is free, then $p$ is étale and
\begin{align*}
\Gamma \times_S X &\to X \times_Y X\\
  (\gamma,x) &\mapsto (x,\gamma\cdot x)
\end{align*}
is an isomorphism.
\end{enumerate}
\end{lemma}

\begin{obs} \label{remfiniteaction}
Part (2) in the above lemma implies that,  when the action of $\Gamma$ on $X$ is free, then the stacky quotient $[X/\Gamma]$ (cf. \cite{olsson16} Example 8.1.12) is representable by the scheme $X/\Gamma$.
\end{obs}
 
% \begin{proof}
% Since $X$ is quasi-projective over $S$, by a homogeneous version of the prime-avoidance lemma, any finite set of points in $X$ is contained in an affine open subscheme of $X$. Therefore, 1 is a direct consequence of \cite{SGA103}, V.1.8.

%  It follows from \cite{SGA103} V.1.3 that $p$ is surjective. Moreover, as any quasi-projective morphism is of finite type, $p$ is finite by \cite{SGA103}, V.1.5. 

%  Quasi-projectiveness of $Y$ is a consequence of the following general fact: if $R$ is a ring, $n \ge 1$ is an integer, and $G$ is a finite abstract group acting linearly on $A=R[x_0,\ldots,x_n]$, then $\Proj A^G$ is projective over $R$. We refer to \cite{knutson71} IV.1 for a proof of 3.

% Finally, to prove 4, we observe that if $\Gamma$ acts freely on $X$, then for any $x\in X$, the inertia group at $x$ (that is, the subgroup of $|\Gamma|$ consisting of those elements that fix $x$ and induces the identity on the residue field $k(x)$) is trivial. Then, by \cite{SGA103} V.2.3, $p$ is étale. The last assertion is obtained by \cite{SGA103} V.2.6.  
% \end{proof}

\begin{proof}[Proof of Theorem \ref{repr}]
Recall from \cite{GIT94} Theorem 7.9 (cf. \cite{olsson12} proof of Theorem 2.1.11) that there exists a quasi-projective scheme $A_{g,1,4}$ over $\ZZ[1/2]$ endowed with an action by the constant finite group scheme $\Gamma = \GL_{g}(\ZZ/4\ZZ)_{\ZZ[1/2]}$ over $\ZZ[1/2]$, and with a surjective étale morphism $A_{g,1,4} \to \mathcal{A}_{g,\ZZ[1/2]}$ inducing an isomorphism of the stacky quotient $[A_{g,1,4}/\Gamma]$ with $\mathcal{A}_{g,\ZZ[1/2]}$; namely, $A_{g,1,4}$ is a scheme representing the functor
\begin{align*}
\mathcal{A}_g^{\opp} &\to \Sets\\
              (X,\lambda)_{/U}&\mapsto \text{Isom}_{\textsf{GpSch}_{/U}}((\ZZ/4\ZZ)_U^{2g}, X[4])\text{.}
\end{align*}

As the morphism of Deligne-Mumford stacks over $\Spec \ZZ$
\begin{align*}
\pi_g : \mathcal{B}_g \to \mathcal{A}_g
\end{align*}
is representable by smooth affine schemes (Remark \ref{relrepr}), the fiber product
\begin{align*}
\mathcal{F} = A_{g,1,4}\times_{\mathcal{A}_{g,\ZZ[1/2]}}\mathcal{B}_{g,\ZZ[1/2]}
\end{align*}
is representable by a smooth affine scheme $B$ over $A_{g,1,4}$ via the first projection $\mathcal{F} \to A_{g,1,4}$. In particular, $B$ is affine and of finite type over $A_{g,1,4}$. Since $A_{g,1,4}$ is quasi-projective over $\ZZ[1/2]$, it follows that $B$ is a quasi-projective $\ZZ[1/2]$-scheme.

The action of $\Gamma$ on $A_{g,1,4}$ naturally induces an action of $\Gamma$ on $\mathcal{F}$, thus on $B$. As $\mathcal{B}_{g,\ZZ[1/2]}$ is an algebraic space by Proposition \ref{rigid} (cf. remark following Definition \ref{defrigid}), this action is free. Moreover, by the compatibility of quotients of stacks by group actions with base change (cf. \cite{romagny05} Proposition 2.6), the second projection $\mathcal{F} \to \mathcal{B}_{g,\ZZ[1/2]}$ induces an isomorphism of the stacky quotient $[B/\Gamma]$ with $\mathcal{B}_{g,\ZZ[1/2]}$.  Finally, by Lemma \ref{finiteaction} and Remark \ref{remfiniteaction}, we conclude that $\mathcal{B}_{g,\ZZ[1/2]}$ is representable by the quasi-projective $\ZZ[1/2]$-scheme $B/\Gamma$.
\end{proof}

\section{The vector bundle $T_{\mathcal{B}_g/\ZZ}$ and the higher Ramanujan vector fields} \label{ramvecfields}

Fix an integer $g\ge 1$. We define a presheaf $\mathcal{H}_g$ (resp. $\mathcal{F}_g$) of $\mathcal{O}_{\mathcal{A}_{g,\et}}$-modules on $\Et(\mathcal{A}_g)$ as follows. Let $(U,u)$ be an étale scheme over $\mathcal{A}_g$, and $(X,\lambda)$ be the principally polarized abelian scheme over $U$ corresponding to $u:U \to \mathcal{A}_g$. We put
\begin{align*}
\Gamma((U,u),\mathcal{H}_g)\defeq \Gamma(U,H^1_{\dR}(X/U)) \ \ \text{(resp. }\Gamma((U,u),\mathcal{F}_g) \defeq \Gamma(U,F^1(X/U))\text{)}
\end{align*}
If $(f,f^b):(U',u') \to (U,u)$ is a morphism in $\Et(\mathcal{A}_g)$, the restriction map is given by the comparison morphism $f^*H^1_{\dR}(X/U) \to H^1_{\dR}(X'/U')$ (resp. $f^*F^1(X/U)\to F^1(X'/U')$), where $(X',\lambda') = (X,\lambda)\times_UU'$. As the comparison morphism is actually an isomorphism (i.e. the formation of $H^1_{\dR}(X/U)$ (resp. $F^1(X/U)$) is compatible with base change), and $H^1_{\dR}(X/U)$ (resp. $F^1(X/U)$) is quasi-coherent, $\mathcal{H}_g$ (resp. $\mathcal{F}_g$) is a quasi-coherent sheaf over $\mathcal{A}_g$ (cf. \cite{olsson16} Lemma 4.3.3). We finally remark that $\mathcal{H}_g$ is actually a vector bundle of rank $2g$ over $\mathcal{A}_g$ and that $\mathcal{F}_g$ is a rank $g$ subbundle of $\mathcal{H}_g$.

\begin{obs}
The sheaf $\mathcal{H}_g$ should be thought as the first de Rham cohomology of the universal abelian scheme over $\mathcal{A}_g$, and $\mathcal{F}_g$ as its Hodge subbundle. Indeed, if $X_g$ denotes the universal abelian scheme over $B_g$ (cf. Theorem \ref{repr}), then, under the isomorphism $\mathcal{B}_{g,\ZZ[1/2]} \cong B_g$, the sheaf $\mathcal{H}_g$ gets identified with $H^1_{\dR}(X_g/B_g)$, and $\mathcal{F}_g$ with $F^1(X_g/B_g)$.
\end{obs}

In this section we describe the tangent bundle $T_{\mathcal{B}_g/\ZZ}$ in terms of $\mathcal{H}_g$ and $\mathcal{F}_g$; see Theorem \ref{theorem1} for a precise statement. This will be obtained by realizing $\mathcal{B}_g$ as a substack of the stack over $\mathcal{A}_g$ associated to the vector bundle $\mathcal{H}_{g}^{\oplus g}$.

Further, Theorem \ref{theorem1} will allow us to construct a certain family of $g(g+1)/2$ global sections $v_{ij}$ of $T_{\mathcal{B}_g/\ZZ}$ that we call \emph{higher Ramanujan vector fields}.

% In Section 6, we shall recall how $\mathcal{B}_{1,\ZZ[1/6]}$ may be identified, by means of the classical theory of elliptic curves, with the open subscheme 

% The reader will find in Section \ref{caseg=1} a proof that, in the case $g=1$, the Ramanujan vector field so defined boils down to the classical Ramanujan vector field defined in the introduction, thereby justifying our terminology.  

\subsection{The Gauss-Manin connection and the Kodaira-Spencer morphism on $\mathcal{A}_{g}/\ZZ$} \label{kodairaspencer}

\subsubsection{}

In order to give a precise statement of Theorem \ref{theorem1}, we need to recall some basic facts concerning Gauss-Manin connections and Kodaira-Spencer morphisms over abelian schemes.

 Fix a base scheme $S$ and let $p:X \to U$ be a projective abelian scheme, with $U$ a \emph{smooth} $S$-scheme. Then there is defined an integrable $S$-connection over the de Rham cohomology sheaves (\cite{KO68}; see also \cite{katz70}), the \emph{Gauss-Manin connection}
\begin{align} \label{GMcon}
\nabla : H^i_{\dR}(X/U) \longrightarrow  H^i_{\dR}(X/U)\tensor_{\mathcal{O}_U} \Omega^1_{U/S}\text{,}
\end{align}
whose formation is compatible with every base change $U'\to U$, where $U'$ is a smooth $S$-scheme. 
 
The Gauss-Manin connection on $H^1_{\dR}(X/U)$ induces a morphism
\begin{align*}
T_{U/S} &\longrightarrow \mathcal{H}om_{\mathcal{O}_S}(H_{\dR}^1(X/U),H^1_{\dR}(X/U))\\
    \theta &\longmapsto \nabla_{\theta}( \ \ )\text{.}
\end{align*}
Restricting to $F^1(X/U)$ and passing to the quotient (cf. exact sequence (\ref{hodgefiltr})), we obtain an $\mathcal{O}_{U}$-morphism 
\begin{align*}
T_{U/S} \longrightarrow &\mathcal{H}om_{\mathcal{O}_U}(F^1(X/U),R^1p_*\mathcal{O}_X)\cong F^1(X/U)^{\vee}\tensor_{\mathcal{O}_U}R^1p_*\mathcal{O}_X\text{.}
\end{align*}
Applying the inverse of the canonical isomorphism $\phi_{X^t/U}^1: F^1(X^t/U)^{\vee} \stackrel{\sim}{\to} R^1p_*\mathcal{O}_{X}$ (cf. proof of Lemma \ref{f1lagrangian}, where we identified $X$ with $X^{tt}$ via the canonical biduality isomorphism), we obtain an $\mathcal{O}_U$-morphism
\begin{align*}
\delta : T_{U/S} \longrightarrow F^1(X/U)^{\vee}\tensor_{\mathcal{O}_U}F^1(X^t/U)^{\vee}\text{.}
\end{align*}
This is, possibly up to a sign, the dual of $\rho$ defined in \cite{FC90} III.9.\footnote{With notations as in the proof of Lemma \ref{f1lagrangian}, there are two natural ways of identifying $R^1p_*\mathcal{O}_X$ with $F^1(X^t/U)^{\vee}$: one by $(\phi_{X/U}^0)^{\vee}$, and another by $\phi_{X^t/U}^1$. These produce the same isomorphisms up to a sign. In \cite{FC90} this choice is not specified.}

\subsubsection{}

 Now, with the same notations and hypothesis as above, let $\lambda :X \to X^t$ be a principal polarization. The Gauss-Manin connection $\nabla$ on $H^1_{\dR}(X/U)$ is compatible with the symplectic form $\langle \ , \ \rangle_{\lambda}$ in the following sense. For every sections $\theta$ of $T_{U/S}$, and $\alpha$ and $\beta$ of $H^1_{\dR}(X/U)$, we have
\begin{align} \label{C}
\theta \langle \alpha, \beta\rangle_{\lambda} = \langle \nabla_{\theta}\alpha,\beta \rangle_{\lambda} + \langle \alpha, \nabla_{\theta}\beta \rangle_{\lambda}\text{.}
\end{align}
This can be deduced from the fact that the first Chern class in $H^2_{\dR}(X\times_UX^t/U)$ of the Poincaré line bundle $\mathcal{P}_{X/U}$ is horizontal for the Gauss-Manin connection, since it actually comes from a class in $H^2_{\dR}(X\times_UX^t/S)$.

 By composing $\delta$ with $((\lambda^*)^{\vee})^{-1} : F^1(X^t/U)^{\vee} \stackrel{\sim}{\to}F^1(X/U)^{\vee}$, we obtain a morphism 
\begin{align} \label{KSmor}
\kappa : T_{U/S} \longrightarrow F^1(X/U)^{\vee}\tensor_{\mathcal{O}_U}F^1(X/U)^{\vee}\text{.}
\end{align}
This is the \emph{Kodaira-Spencer morphism} associated to $(X,\lambda)_{/U}$ over $S$. It follows from the compatibility (\ref{C}) that $\kappa$ factors through the submodule of \emph{symmetric tensors} in $F^1(X/U)^{\vee}\tensor_{\mathcal{O}_U}F^1(X/U)^{\vee}$, denoted  $\Gamma^2(F^1(X/U)^{\vee})$.

\begin{obs} \label{explicite}
As $\phi_{X^t/U}^{\vee} = -\phi_{X/U}$ under the canonical biduality isomorphism $X\cong X^{tt}$ (cf. \cite{BBM82} Lemme 5.1.5), one may verify that the composition
\begin{align*}
R^1p_*\mathcal{O}_X \stackrel{(\phi^1_{X^t/U})^{-1}}{\to} F^1(X^t/U)^{\vee} \stackrel{((\lambda^*)^{\vee})^{-1}}{\to} F^1(X/U)^{\vee}
\end{align*}
considered above is given by the isomorphism of vector bundles $H^1_{\dR}(X/U)/F^1(X/U) \stackrel{\sim}{\to} F^1(X/U)^{\vee}$ induced by (cf. Lemma \ref{exactseq})
\begin{align*}
H^1_{\dR}(X/U) &\to H^1_{\dR}(X/U)^{\vee}\\
         \alpha & \mapsto \langle \  \ , \alpha\rangle_{\lambda}\text{.}
\end{align*}
Thus, if $b = (\omega_1,\ldots,\omega_g,\eta_1,\ldots,\eta_g)$ is a symplectic-Hodge basis of $(X,\lambda)_{/U}$,  $\kappa$ admits the following explicit description in terms of $b$:
% \begin{align*}
%   \kappa (\theta) = \sum_{i=1}^g \langle \ \ , \eta_i \rangle_{\lambda}\tensor\langle \omega_i,\nabla_{\theta}( \ \ )\rangle_{\lambda} \text{.}
% \end{align*}
\begin{align*}
  \kappa (\theta) = \sum_{i=1}^g \langle \ \ , \eta_i \rangle_{\lambda}\tensor\langle \ \ ,\nabla_{\theta}\omega_i\rangle_{\lambda} \text{.}
\end{align*}
\end{obs}

 Finally, we remark that the Kodaira-Spencer morphism is natural in the following sense. Let $U'$ be a smooth scheme over $S$ and let $F_{/f}:(X',\lambda')_{/U'} \to (X,\lambda)_{/U'}$ be a morphism in $\mathcal{A}_{g,S}$. Denote by $\kappa$ (resp. $\kappa'$) the Kodaira-Spencer morphism associated to $(X,\lambda)_{/U}$ (resp.  $(X',\lambda')_{/U'}$) over $S$. Then the diagram
$$
\raisebox{-0.5\height}{\includegraphics{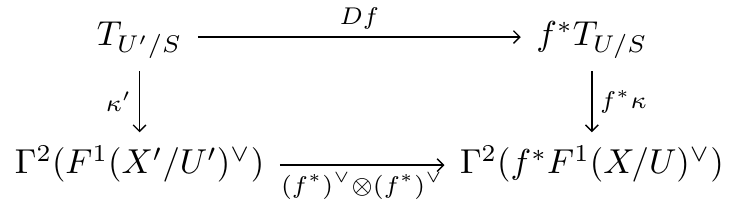}}
$$
% $$
% \begin{tikzcd}[column sep = huge]
% T_{U'/S} \arrow{r}{Df}\arrow{d}[swap]{\kappa'} &f^*T_{U/S}\arrow{d}{f^*\kappa} \\
% \Gamma^2(F^1(X'/U')^{\vee}) \arrow{r}[swap]{(f^*)^{\vee}\tensor(f^*)^{\vee}} & \Gamma^2(f^*F^1(X/U)^{\vee})
% \end{tikzcd}
% $$
commutes. 

\subsubsection{} \label{ksiso}

Let $S$ be a scheme, and denote by $\mathcal{H}_{g,S}$ (resp. $\mathcal{F}_{g,S}$) the vector bundle over $\mathcal{A}_{g,S}$ obtained from $\mathcal{H}_g$ (resp. $\mathcal{F}_g$) by the base change $\mathcal{A}_{g,S}\to \mathcal{A}_g$. As $\mathcal{A}_{g,S}\to S$ is smooth, the naturality of the Gauss-Manin connection permits us to construct a ``universal'' Gauss-Manin connection
\begin{align*}
\nabla : \mathcal{H}_{g,S} \to \mathcal{H}_{g,S}\tensor_{\mathcal{O}_{\mathcal{A}_{g,S,\et}}}\Omega^1_{\mathcal{A}_{g,S}/S}
\end{align*}
and the naturality of the Kodaira-Spencer morphism permits us to construct a ``universal'' Kodaira-Spencer morphism
\begin{align*}
\kappa : T_{\mathcal{A}_{g,S}} \longrightarrow \Gamma^2(\mathcal{F}_{g,S}^{\vee})\text{.}
\end{align*}
These are morphism of sheaves on the étale site of $\mathcal{A}_{g,S}$ given, for any étale scheme $(U,u)$ over $\mathcal{A}_{g,S}$ corresponding to the principally polarized abelian scheme $(X,\lambda)$ over the $S$-scheme $U$, respectively by the Gauss-Manin connection (\ref{GMcon}) and the Kodaira-Spencer morphism of $(X,\lambda)_{/U}$ over $S$ (\ref{KSmor}); note that as $u :U \to \mathcal{A}_{g,S}$ is étale, then $U$ is smooth over $S$.

We remark that the universal Kodaira-Spencer morphism $\kappa : T_{\mathcal{A}_{g,S}} \to \Gamma^2(\mathcal{F}_{g,S}^{\vee})$ is actually an \emph{isomorphism} of $\mathcal{O}_{\mathcal{A}_{g,S,\et}}$-modules (cf. \cite{FC90} Theorem 5.7.(3)).

Finally, let $\mathcal{U}$ be a smooth Deligne-Mumford stack over $S$ and $u : \mathcal{U} \to \mathcal{A}_{g,S}$ be a quasi-compact and quasi-separated morphism of $S$-stacks representable by schemes. Then, the Gauss-Manin connection over $(\mathcal{U},u)$, or simply over $\mathcal{U}$ if $u$ is implicit, 
\begin{align*}
\nabla : u^*\mathcal{H}_{g,S} \to u^*\mathcal{H}_{g,S}\tensor_{\mathcal{O}_{\mathcal{U},\et}} \Omega^1_{\mathcal{U}/S}
\end{align*} 
is defined by pulling back the universal Gauss-Manin connection on $\mathcal{A}_{g,S}$. Further, we may define a Kodaira-Spencer morphism over $(\mathcal{U},u)$ as the composition
\begin{align*}
\kappa_{u} : T_{\mathcal{U}/S} \stackrel{D u }{\longrightarrow} u^*T_{\mathcal{A}_{g,S}/S} \stackrel{u^*\kappa}{\longrightarrow} \Gamma^2(u^*\mathcal{F}_{g,S}^{\vee})\text{.}
\end{align*}
 
\subsection{The embedding $i_g: \mathcal{B}_g \to \mathcal{V}_g$}

We shall employ the following notations in the statement of Theorem \ref{theorem1}. Consider the morphism of coherent $\mathcal{O}_{\mathcal{B}_{g,\et}}$-modules
\begin{align*}
m_g:\pi_g^*\mathcal{H}_g^{\oplus g} \longrightarrow M_{g\times g}(\mathcal{O}_{\mathcal{B}_{g,\et}})
\end{align*}
defined as follows. Let $(U,u)$ be an étale scheme over $\mathcal{B}_g$, and let $(X,\lambda,b)_{/U}$, $b=(\omega_1,\ldots,\omega_g,\eta_1,\ldots,\eta_g)$, be the corresponding object of the fiber category $\mathcal{B}_g(U)$. Then the morphism $m_g$ sends a section $(\alpha_1,\ldots,\alpha_g)$ of $H^1_{\dR}(X/U)^{\oplus g}$ to the section
\begin{align} \label{matrice}
(\langle \alpha_i,\eta_j\rangle_{\lambda})_{1\le i,j\le g} 
\end{align}
of $M_{g\times g}(\mathcal{O}_U)$.  We can thus define a subbundle $\mathcal{S}_g$ of $\pi_g^*\mathcal{H}_g^{\oplus g}$ as the inverse image of the subbundle of symmetric matrices $\Sym_g(\mathcal{O}_{\mathcal{B}_{g,\et}})$ by this morphism. In other words, if $(U,u)$ and $(X,\lambda,b)_{/U}$ are as above, a section $(\alpha_1,\ldots,\alpha_g)$ of $H^1_{\dR}(X/U)^{\oplus g}$ is in $\mathcal{S}_g$ if and only if the matrix (\ref{matrice}) is symmetric. 

\begin{obs} \label{remarkrank}
Note that $m_g$ is surjective: with the above notations, for a given matrix $(a_{ij})_{1\le i,j\le g}$ in $M_{g\times g}(\mathcal{O}_U)$, take $\alpha_i = \sum_{j=1}^ga_{ij}\omega_j$. In particular, $\mathcal{S}_g$ is a subbundle of $\pi_g^*\mathcal{H}_g^{\oplus g}$ of rank $g^2 + g(g+1)/2 = g(3g+1)/2$.
\end{obs}

\begin{theorem} \label{theorem1}
Consider the morphism of quasi-coherent $\mathcal{O}_{\mathcal{B}_{g,\et}}$-modules
\begin{align*}
c_g:T_{\mathcal{B}_g/\ZZ} \longrightarrow \Gamma^2(\pi_g^*\mathcal{F}_g^{\vee}) \oplus \pi_g^*\mathcal{H}_g^{\oplus g}
\end{align*}
defined by
\begin{align*}
c_g(\theta) = (\kappa_{u}(\theta),\nabla_{\theta}\eta_1, \ldots,\nabla_{\theta}\eta_g)
\end{align*}
for every étale scheme $(U,u)$ over $\mathcal{B}_g$ corresponding to the object $(X,\lambda,b)_{/U}$ of $\mathcal{B}_g(U)$, where $b=(\omega_1,\ldots,\omega_g,\eta_1,\ldots,\eta_g)$, and $\theta$ a section of $T_{U/\ZZ}$. Then $c_g$ induces an isomorphism of $T_{\mathcal{B}_g/\ZZ}$ onto the subbundle $\Gamma^2(\pi_g^*\mathcal{F}_g^{\vee})\oplus \mathcal{S}_g$ of $\Gamma^2(\pi_g^*\mathcal{F}_g^{\vee}) \oplus \pi_g^*\mathcal{H}_g^{\oplus g}$.
\end{theorem}

A proof of this result will be given at the end of this paragraph.

\subsubsection{} \label{vg}

Consider the associated space of the vector bundle $\mathcal{H}_g^{\oplus g}$ (cf. \cite{olsson16} 10.2)
\begin{align*}
\mathcal{V}_g \defeq \mathbf{V}((\mathcal{H}_g^{\oplus g})^{\vee})=\underline{\Spec}_{\mathcal{A}_g}(\mathcal{S}ym(\mathcal{H}_g^{\oplus g})^{\vee})\text{.}
\end{align*}
This is a Deligne-Mumford stack over $\Spec \ZZ$ whose objects lying over a scheme $U$ are given by ``$(g+2)$-uples''
\begin{align*}
(X, \lambda, \alpha_1,\ldots,\alpha_g)_{/U}\text{,}
\end{align*}
where $(X,\lambda)_{/U}$ is an object of $\mathcal{A}_{g}(U)$, and $\alpha_i$ is a global section of $H^1_{\dR}(X/U)$ for every $1\le i\le g$. Note that the forgetful functor
\begin{align*}
\Phi_g : \mathcal{V}_g \to \mathcal{A}_{g}
\end{align*}
defines a morphism of stacks representable by smooth affine schemes. 

We define a morphism of stacks
\begin{align*}
i_g : \mathcal{B}_g \to \mathcal{V}_g
\end{align*}
 as follows. Let $(X,\lambda,b)_{/U}$ be an object of $\mathcal{B}_g$ and denote $b=(\omega_1,\ldots,\omega_g,\eta_1,\ldots,\eta_g)$. Then $i_g$ sends $(X,\lambda,b)_{/U}$ to the object
\begin{align*}
(X,\lambda,\eta_1,\ldots,\eta_g)_{/U}
\end{align*}
of $\mathcal{V}_g$. The action of $i_g$ on morphisms is evident. Note that the diagram of morphisms of stacks
$$
\raisebox{-0.5\height}{\includegraphics{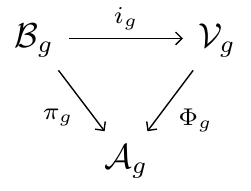}}
$$
% $$
% \begin{tikzcd}[column sep=tiny]
% \mathcal{B}_g\arrow{rd}[swap]{\pi_g} \arrow{rr}{i_g} & & \mathcal{V}_g \arrow{dl}{\Phi_g}\\
% & \mathcal{A}_g
% \end{tikzcd}
% $$
is (strictly) commutative.

\begin{lemma} \label{immersionlemma}
The morphism $i_g: \mathcal{B}_g\to \mathcal{V}_g$ is an immersion of stacks.
\end{lemma}

\begin{proof}
Let $U$ be a scheme and $U \to \mathcal{V}_g$ be a morphism corresponding to the object $(X,\lambda, \alpha_1,\ldots,\alpha_g)_{/U}$ of $\mathcal{V}_g(U)$. Then the fiber product $\mathcal{B}_g \times_{\mathcal{V}_g} U$ can be naturally identified with the locally closed subscheme of $U$ defined by the equations
\begin{align*}
\overline{\alpha}_1\wedge \cdots \wedge \overline{\alpha}_g &\neq 0\\
\langle \alpha_i,\alpha_j \rangle_{\lambda} &= 0 \text{, }\ \ \forall i,j
\end{align*}
where $\overline{\alpha}_i$ denotes the image of $\alpha_i$ in $H^1_{\dR}(X/U)/F^1(X/U)$ (cf. Proposition \ref{exisunic} (2)).
\end{proof}

\subsubsection{} \label{ehresmann}

The proof Theorem \ref{theorem1} relies on Ehresmann's point of view on connections on vector bundles. Let us briefly recall how this goes in our context.

Let $S$ be a scheme, $X$ be a smooth $S$-scheme and $\mathcal{E}$ be a vector bundle over $X$. We denote by $E = \mathbf{V}(\mathcal{E}^{\vee})$ the associated space and by $p:E \to X$ the projection morphism. As $p$ is smooth, we have the exact sequence of vector bundles over $E$
\begin{align*}
0 \longrightarrow T_{E/X} \longrightarrow T_{E/S} \stackrel{Dp}{\longrightarrow}  p^*T_{X/S} \to 0\text{.}
\end{align*}
We claim that every $S$-connection $\nabla :\mathcal{E} \to \mathcal{E}\tensor_{\mathcal{O}_X} \Omega^1_{X/S}$ induces a canonical splitting of the above exact sequence. In fact, this can be obtained by means of the projection
\begin{align*}
P_{\nabla}: T_{E/S} \longrightarrow T_{E/X}
\end{align*} 
defined as follows. The vector bundle $T_{E/X}$ is canonically isomorphic to $p^*\mathcal{E}$ (\cite{EGAIV4} Corollaire 16.4.9); it is thus endowed with a universal global section, say $s$. We put $P_{\nabla}(\theta)= (p^*\nabla)_{\theta}s$.

It is not difficult to transpose the above considerations to the case of smooth Deligne-Mumford stacks (cf. \ref{tangentstacks}).

\subsubsection{} \emph{Proof of Theorem \ref{theorem1}}. Let $\mathcal{V}_g$ and $\Phi_g: \mathcal{V}_g \to \mathcal{A}_g$ be as in \ref{vg}. According to the discussion in \ref{ehresmann},  the connection on $\mathcal{H}_g^{\oplus g}$ given by the direct sum of the ``universal'' Gauss-Manin connection $\nabla : \mathcal{H}_g \to \mathcal{H}_g \tensor_{\mathcal{O}_{\mathcal{A}_{g,\et}}} \Omega^1_{\mathcal{A}_g/\ZZ}$ at each factor induces a splitting  of the exact sequence
\begin{align*}
0 \longrightarrow T_{\mathcal{V}_g/\mathcal{A}_{g}} \longrightarrow T_{\mathcal{V}_g/\ZZ} \stackrel{D\Phi_g}{\longrightarrow} \Phi_g^*T_{\mathcal{A}_{g}/\ZZ} \longrightarrow 0\text{.}
\end{align*}
Thus, after identifying $T_{\mathcal{V}_g/\mathcal{A}_{g}}$ with $\Phi_g^*\mathcal{H}_g^{\oplus g}$, we obtain an isomorphism
\begin{align*}
\overline{c}_g' : T_{\mathcal{V}_{g}} \stackrel{\sim}{\longrightarrow} \Phi_g^*T_{\mathcal{A}_{g}/\ZZ} \oplus \Phi_g^*\mathcal{H}_g^{\oplus g}
\end{align*}
given explicitly by
\begin{align*}
\overline{c}_g'(\theta) = (D\Phi_g(\theta),\nabla_{\theta}\alpha_1,\ldots,\nabla_{\theta}\alpha_g)
\end{align*}
for every étale scheme  $(U,u)$ over $\mathcal{V}_g$ corresponding to the object $(X,\lambda,\alpha_1,\ldots,\alpha_g)_{/U}$ of $\mathcal{V}_g(U)$, and every section $\theta$ of $T_{U/\ZZ}$.

 By composing $\overline{c}_g'$ with the Kodaira-Spencer isomorphism $\kappa : T_{\mathcal{A}_g/\ZZ} \stackrel{\sim}{\to} \Gamma^2(\mathcal{F}_g^{\vee})$ (see \ref{ksiso}), we obtain an isomorphism
\begin{align*}
\overline{c}_g : T_{\mathcal{V}_g/\ZZ} \stackrel{\sim}{\longrightarrow} \Gamma^2(\Phi_g^*\mathcal{F}_g^{\vee}) \oplus \Phi_g^*\mathcal{H}_g^{\oplus g}
\end{align*}
given explicitly by
\begin{align*}
\overline{c}_g(\theta) = (\kappa_{u}(\theta),\nabla_{\theta}\alpha_1,\ldots,\nabla_{\theta}\alpha_g) 
\end{align*}
with notations as above. Finally, note that the morphism $c_g$ in the statement is defined by restricting $\overline{c}_g$ to $\mathcal{B}_g$ via the immersion $i_g : \mathcal{B}_g \to \mathcal{V}_g$ (cf. Lemma \ref{immersionlemma}). In particular, as $\mathcal{B}_g$ is a smooth substack of $\mathcal{V}_g$ via $i_g$, then $c_g$ induces an isomorphism of $T_{\mathcal{B}_g/\ZZ}$ onto a subbundle, say $\mathcal{E}_g$, of $\Gamma^2(\pi_g^*\mathcal{F}_g^{\vee}) \oplus \pi_g^*\mathcal{H}_g^{\oplus g}$. To finish the proof, it is sufficient to show that $\mathcal{E}_g = \Gamma^2(\pi_g^*\mathcal{F}_g^{\vee})\oplus \mathcal{S}_g$. 

Note that the compatibility (\ref{C}) between Gauss-Manin connections and principal polarizations implies that $c_g$ factors by the subbundle $\Gamma^2(\pi_g^*\mathcal{F}_g^{\vee})\oplus \mathcal{S}_g$. Indeed, let $(U,u)$ be an étale scheme over $\mathcal{B}_g$ corresponding to the object $(X,\lambda,b)_{/U}$ of $\mathcal{B}_g(U)$, with $b=(\omega_1,\ldots,\omega_g,\eta_1,\ldots,\eta_g)$, and let $\theta$ be a section of $T_{U/\ZZ}$. Then, as $\langle \eta_i,\eta_j\rangle_{\lambda}=0$, we obtain
\begin{align*}
0 = \nabla_{\theta}\langle \eta_i,\eta_j\rangle_{\lambda} = \langle \nabla_{\theta}\eta_i,\eta_j \rangle_{\lambda} + \langle \eta_i,\nabla_{\theta}\eta_j\rangle_{\lambda} = \langle \nabla_{\theta}\eta_i,\eta_j \rangle_{\lambda} - \langle \nabla_{\theta}\eta_j,\eta_i \rangle_{\lambda}\text{.}
\end{align*}

This proves that $\mathcal{E}_g$ is a subbundle of $\Gamma^2(\pi_g^*\mathcal{F}_g^{\vee})\oplus \mathcal{S}_g$. To conclude, we simply remark that the ranks of the subbundles $\mathcal{E}_g$ and $\Gamma^2(\pi_g^*\mathcal{F}_g^{\vee})\oplus \mathcal{S}_g$ of $\Gamma^2(\pi_g^*\mathcal{F}_g^{\vee})\oplus \pi_g^*\mathcal{H}_g^{\oplus g}$ coincide (cf. Remark \ref{remarkrank}). \hfill $\blacksquare$

\subsection{The higher Ramanujan vector fields} 

Let
\begin{align*}
\langle \  , \ \rangle : \pi_g^*\mathcal{H}_g \times \pi_g^*\mathcal{H}_g \to \mathcal{O}_{\mathcal{B}_{g,\et}}
\end{align*}
be the symplectic $\mathcal{O}_{\mathcal{B}_{g,\et}}$-bilinear form given, for each étale scheme $(U,u)$ over $\mathcal{B}_g$ corresponding to the object $(X,\lambda,b)_{/U}$ of $\mathcal{B}_g(U)$, by
\begin{align*}
u^*\langle \ , \ \rangle \defeq  \langle \ , \ \rangle_{\lambda} : H^1_{\dR}(X/U) \times H^1_{\dR}(X/U) \to \mathcal{O}_{U}
\end{align*}
This is well-defined by Remark \ref{remarkbasechange}.

We denote by
\begin{align*}
b_g = (\omega_1,\ldots,\omega_g,\eta_1,\ldots,\eta_g)
\end{align*}
the ``universal'' symplectic-Hodge basis over $\mathcal{B}_g$. Namely, $b_g$ is the basis of the vector bundle $\pi_g^*\mathcal{H}_g$ such that for every étale scheme $(U,u)$ over $\mathcal{B}_g$ corresponding to the object $(X,\lambda,b)_{/U}$ of $\mathcal{B}_g(U)$ we have $u^*b_g=b$.

Note that vector bundle $\pi_g^*\mathcal{F}_g^{\vee}$ is trivialized over $\mathcal{B}_g$ by the global sections
\begin{align*}
\langle \ \ , \eta_i \rangle : \pi_g^*\mathcal{F}_g \to \mathcal{O}_{\mathcal{B}_{g,\et}}
\end{align*}
for $1\le i \le g$ (this is the dual basis of $(\omega_1,\ldots,\omega_g)$). Accordingly, $\Gamma^2(\pi_g^*\mathcal{F}_g^{\vee})$ is trivialized by the global sections
\begin{align*}
\varphi_{ij} \defeq \begin{cases}
               \langle \ \ , \eta_i \rangle \tensor \langle \ \ , \eta_i \rangle & i=j\\
               \langle \ \ , \eta_i \rangle\tensor \langle \ \ , \eta_j \rangle + \langle \ \ , \eta_j \rangle\tensor \langle \ \ , \eta_i \rangle & i <j
              \end{cases}
\end{align*}
for $1\le i\le j \le g$.

Let $c_g: T_{\mathcal{B}_g/\ZZ} \stackrel{\sim}{\to} \Gamma^2(\pi_g^*\mathcal{F}_g^{\vee}) \oplus \mathcal{S}_g$ be the isomorphism of $\mathcal{O}_{\mathcal{B}_{g,\et}}$-modules defined in Theorem \ref{theorem1}.

\begin{defi} \label{deframvf}
For every $1\le i \le j \le g$, we define the \emph{higher Ramanujan vector field} $v_{ij}$ as being the unique global section of $T_{\mathcal{B}_g/\ZZ}$ such that $c_g(v_{ij})=(\varphi_{ij},0)$.
\end{defi}

Let us denote the ``universal'' Gauss-Manin connection over $\mathcal{B}_g$ by (cf. \ref{ksiso})
\begin{align*}
\nabla : \pi^*_g\mathcal{H}_g \longrightarrow \pi^*_{g}\mathcal{H}_g \tensor_{\mathcal{O}_{\mathcal{B}_{g,\et}}} \Omega^1_{\mathcal{B}_g/\ZZ}\text{.}
\end{align*}

\begin{prop} \label{caracchamps}
The higher Ramanujan vector fields are the unique global sections $v_{ij}$ of $T_{\mathcal{B}_g/\ZZ}$ such that, for every $1\le i \le j \le g$,
\begin{enumerate}
   \item $\nabla_{v_{ij}}\omega_i = \eta_j$, $\nabla_{v_{ij}}\omega_j=\eta_i$, and $\nabla_{v_{ij}}\omega_k=0$ for $k \not\in\{i,j\}$.
   \item $\nabla_{v_{ij}}\eta_k =0$, for every $1\le k \le g$.
\end{enumerate}
\end{prop}

\begin{proof}
The vector fields $v_{ij}$ satisfy (2) by definition of $c_g$ in Theorem \ref{theorem1}. Moreover, using the explicit expression of the Kodaira-Spencer morphism in Remark \ref{explicite}, we see that 
\begin{align} \label{equiv1}
\sum_{k=1}^g\langle \ \ , \eta_k \rangle \tensor \langle \ \ ,\nabla_{v_{ij}}\omega_k \rangle  =  \begin{cases}
               \langle \ \ , \eta_i\rangle  \tensor \langle \ \ , \eta_i \rangle & i=j\\
               \langle \ \ , \eta_i\rangle  \tensor \langle \ \ , \eta_j \rangle +  \langle \ \ , \eta_j\rangle  \tensor \langle \ \ , \eta_i \rangle & i <j
              \end{cases}
\end{align}
in $\Gamma^2(\pi_g^*\mathcal{F}_g^{\vee})$ for every $1\le i \le j \le g$. As $b_g$ is symplectic with respect to $\langle \ , \ \rangle$, by evaluating the second factors at $\eta_l$ for every $1\le l \le g$ in the above equation, we see that $\nabla_{v_{ij}}\omega_k$ lies in the subbundle of $\pi^*_g\mathcal{H}_g$ generated by $\eta_1,\ldots,\eta_g$, for every $1\le i\le j \le g$ and $1\le k \le g$.

Thus, to prove that the vector fields $v_{ij}$ satisfy (1), it is sufficient to prove that
\begin{align}
 \langle \omega_l, \nabla_{v_{ij}}\omega_i\rangle = \delta_{lj}\text{, }\langle \omega_l, \nabla_{v_{ij}}\omega_j\rangle = \delta_{li}\text{, and } \langle \omega_l, \nabla_{v_{ij}}\omega_k\rangle = 0\text{ for }k\not\in\{i,j\}
\end{align}  
for every $1\le l \le g$. This in turn follows immediately from (\ref{equiv1}) by evaluating the second factors at $\omega_l$.

To prove unicity, let $(w_{ij})_{1\le i \le j \le g}$ be a family of vector fields on $\mathcal{B}_g$ satisfying (1) and (2). Note that, by the explicit expression of the Kodaira-Spencer morphism in Remark \ref{explicite}, equations in (1) imply $\kappa_{\pi_g}(w_{ij}) = \varphi_{ij}$ in the notation preceding Definition \ref{deframvf}. Thus, by (1) and (2),
\begin{align*}
c_g(w_{ij}) = (\varphi_{ij},0) = c_g(v_{ij})
\end{align*}
for every $1\le i \le j \le g$, that is, $w_{ij}= v_{ij}$ (cf. Definition \ref{deframvf}).
\end{proof}

Let $S$ be a scheme. We denote by $\pi_{g,S}: \mathcal{B}_{g,S} \to \mathcal{A}_{g,S}$ the base change of $\pi_{g}:\mathcal{B}_{g} \to \mathcal{A}_g$ by $\mathcal{A}_{g,S} \to \mathcal{A}_g$, and by
\begin{align*}
\nabla : \pi_{g,S}^*\mathcal{H}_{g,S} \to \pi_{g,S}^*\mathcal{H}_{g,S} \tensor_{\mathcal{O}_{\mathcal{B}_{g,S,\et}}}\Omega^1_{\mathcal{B}_{g,S}/S}
\end{align*}
the ``universal'' Gauss-Manin $S$-connection over $\mathcal{B}_{g,S}$.

\begin{obs} \label{remarkrg}
Let $\mathcal{R}_{g,S}$ be the $\mathcal{O}_{\mathcal{B}_{g,S,\et}}$-submodule of $T_{\mathcal{B}_{g,S}/S}$ generated by all the $v_{ij}$; if $S=\Spec \ZZ$, we denote simply $\mathcal{R}_{g,S}\eqdef \mathcal{R}_{g}$. It is clear from Theorem \ref{theorem1} that $\mathcal{R}_{g,S}$ is the kernel of the surjective $\mathcal{O}_{\mathcal{B}_{g,S,\et}}$-morphism
\begin{align*}
T_{\mathcal{B}_{g,S}/S} &\longrightarrow \mathcal{S}_g \\
           \theta &\mapsto (\nabla_{\theta} \eta_1,\ldots,\nabla_{\theta} \eta_g)
\end{align*}
 In particular, $\mathcal{R}_{g,S}$ is a subbundle of $T_{\mathcal{B}_{g,S}/S}$ of rank $g(g+1)/2$.
\end{obs}

\begin{lemma}\label{corocommute1}
Let $\theta$ be a section of $T_{\mathcal{B}_{g,S}/S}$ such that $\nabla_{\theta}\omega_i=\nabla_{\theta}\eta_i=0$ for every $1\le i \le g$. Then $\theta=0$.
\end{lemma}

\begin{proof}
Let $\theta$ be as in the statement. By Remark \ref{remarkrg}, $\theta$ is in the subbundle $\mathcal{R}_{g,S}$ of $T_{\mathcal{B}_{g,S}/S}$, thus there exist sections $(f_{ij})_{1\le i \le j \le g}$ of $\mathcal{O}_{\mathcal{B}_{g,S,\et}}$ such that
\begin{align*}
\theta = \sum_{1\le i \le j \le g}f_{ij}v_{ij}\text{.}
\end{align*}
We prove that each $f_{ij}=0$ by induction on $i$. For $i=1$, we have by Proposition \ref{caracchamps}
\begin{align*}
0 = \nabla_{\theta}\omega_1 = \sum_{1 \le i \le j \le g}f_{ij}\nabla_{v_{ij}}\omega_1 = \sum_{j=1}^gf_{1j}\eta_{j}\text{,} 
\end{align*}
thus $f_{1j}=0$ for every $1\le j \le g$. If $2\le i_0 \le g$ and $f_{ij}=0$ for every $i< i_0$ and $i\le j \le g$, we have
\begin{align*}
0 = \nabla_{\theta}\omega_{i_0} = \sum_{i_0 \le i \le j \le g}f_{ij}\nabla_{v_{ij}}\omega_{i_0}= \sum_{j=i_0}^gf_{i_0j}\eta_j\text{,} 
\end{align*}
thus $f_{i_0j}=0$ for every $i_0\le j \le g$.
\end{proof}

Let $[ \ , \ ]$ denote the Lie bracket in $T_{\mathcal{B}_g/\ZZ}$.

\begin{coro} \label{corocommute}
 The higher Ramanujan vector fields commute. That is,
\begin{align*}
[v_{ij},v_{i'j'}] = 0
\end{align*}
for any $1\le i\le j \le g$ and $1\le i'\le j' \le g$.
\end{coro} 

\begin{proof}
Let us first remark that, as the Gauss-Manin connection is integrable, for any sections $\theta$ and $\theta'$ of $T_{\mathcal{B}_g/\ZZ}$, we have
\begin{align*}
\nabla_{[\theta,\theta']} = \nabla_{\theta} \nabla_{\theta'} - \nabla_{\theta'} \nabla_{\theta}\text{.}
\end{align*}
This implies that $\mathcal{R}_g$ is integrable: if $\theta$ and $\theta'$ are both sections of $\mathcal{R}_g$, then $[\theta,\theta']$ is a section of $\mathcal{R}_g$. In particular, $\theta\defeq [v_{ij},v_{i'j'}]$ is a section of $\mathcal{R}_g$. By Lemma \ref{corocommute1}, to prove that $\theta=0$, it is sufficient to prove that $\nabla_{\theta}\omega_k=0$ for every $1\le k \le g$. 

We have
\begin{align*}
\nabla_{\theta} \omega_k  = \nabla_{v_{ij}}(\nabla_{v_{i'j'}}\omega_k) - \nabla_{v_{i'j'}}(\nabla_{v_{ij}}\omega_k)\text{.}
\end{align*}
It follows from Proposition \ref{caracchamps} that $\nabla_{v_{i'j'}}\omega_k$ (resp. $\nabla_{v_{ij}}\omega_k$) is an element of $\{0,\eta_1,\ldots,\eta_g\}$; hence $\nabla_{v_{ij}}(\nabla_{v_{i'j'}}\omega_k) = 0$ (resp. $\nabla_{v_{i'j'}}(\nabla_{v_{ij}}\omega_k) = 0$).
\end{proof}

% \begin{obs}
% For $1\le i\le j \le g$, let $v_{ij}^{\vee}$ be the global section of $\Omega^1_{\mathcal{B}_g/\ZZ}$ dual to $v_{ij}$. By defining $v_{ji}^{\vee} = v_{ij}^{\vee}$, we obtain a global section $v^\vee = (v_{ij}^{\vee})_{1\le i,j\le g}$ of $M_{g\times g}(\Omega^1_{\mathcal{B}_g/\ZZ})$. Then, in matricial notation (see \ref{matricialnotation}), if $\omega = (\omega_1 \  \cdots \ \omega_g)$ and $\eta = (\eta_1 \ \cdots \ \eta_g)$ are considered as row vectors of order $g$, the conditions in the corollary above are equivalent to $\nabla \omega\transp = v^{\vee}\tensor \eta\transp$ and $\nabla \eta\transp = 0$.
% \end{obs}

\subsection{The action of Siegel parabolic subgroup $P_g$ on the higher Ramanujan vector fields}

Geometrically, $\pi_g:\mathcal{B}_g \to \mathcal{A}_g$ may be regarded as a ``principal $P_g$-bundle'' over $\mathcal{A}_g$ (cf. Lemma \ref{torsor}). It is therefore natural to ask how the integrable subbundle $\mathcal{R}_g$ of $T_{\mathcal{B}_g/\ZZ}$ (cf. Remark \ref{remarkrg} and Corollary \ref{corocommute}) transforms under the action of $P_g$.

In order to formulate precise statements, fix an affine base scheme $S = \Spec R$ and let $p \in P_g(S)$. Then, $p$ induces an $S$-automorphism of $\mathcal{B}_{g,S}$ given by
\begin{align*}
p: \mathcal{B}_{g,S} &\to \mathcal{B}_{g,S}\\
                (X,\lambda,b)_{/U} &\mapsto (X,\lambda,b\cdot p)_{/U}
\end{align*}
where we have implicitly identified $p$ with its image by the natural map $P_g(S)\to P_g(U)$ to compute $b\cdot p$. 

\begin{prop} \label{B=0}
Let us write
\begin{align*}
p = \left(\begin{array}{cc} A & B \\ 0 & (A\transp)^{-1} \end{array} \right)\in P_g(S)\text{,}
\end{align*}
and consider the tangent map
\begin{align*}
Dp : T_{\mathcal{B}_{g,S}/S} \to p^*T_{\mathcal{B}_{g,S}/S}\text{.}
\end{align*}
Then $Dp(\mathcal{R}_{g,S})\subset p^*\mathcal{R}_{g,S}$ if and only if $B=0$.
\end{prop}

Let us introduce some preliminary notation before proving this result. Note that $Dp$ induces an $R$-automorphism
\begin{align*}
p_*: \Gamma(\mathcal{B}_{g,S},T_{\mathcal{B}_{g,S}/S}) \to \Gamma(\mathcal{B}_{g,S},T_{\mathcal{B}_{g,S}/S})
\end{align*} 
which is compatible with the ``universal'' Gauss-Manin $S$-connection
\begin{align*}
\nabla : \pi_{g,S}^*\mathcal{H}_{g,S} \to \pi_{g,S}^*\mathcal{H}_{g,S} \tensor_{\mathcal{O}_{\mathcal{B}_{g,S,\et}}}\Omega^1_{\mathcal{B}_{g,S}/S}
\end{align*}
in the following sense. Denote by
\begin{align*}
p^* : \Gamma(\mathcal{B}_{g,S},\pi_{g,S}^*\mathcal{H}_{g,S}) \to \Gamma(\mathcal{B}_{g,S},\pi_{g,S}^*\mathcal{H}_{g,S})
\end{align*}
the $R$-automorphism induced by the isomorphism of vector bundles $p^*\pi_{g,S}^*\mathcal{H}_{g,S} \stackrel{\sim}{\to}\pi_{g,S}^* \mathcal{H}_{g,S}$ (observe that $\pi_{g,S}\circ p = \pi_{g,S}$).  Then, for any $\alpha \in \Gamma(\mathcal{B}_{g,S},\pi_{g,S}^*\mathcal{H}_{g,S})$, and any $\theta \in \Gamma(\mathcal{B}_{g,S},T_{\mathcal{B}_{g,S}/S})$, we have
\begin{align} \label{comppbpf}
p^*(\nabla_{p_*\theta}\alpha) = \nabla_{\theta}(p^*\alpha)\text{.}
\end{align}

\begin{obs} \label{h*carac}
The automorphism $p^*$ introduced above is characterized by
\begin{align*}
(\begin{array}{cccccc}p^*\omega_1 & \cdots & p^*\omega_g & p^*\eta_1 & \cdots & p^*\eta_g\end{array}) =
(\begin{array}{cccccc}\omega_1 & \cdots & \omega_g & \eta_1 & \cdots & \eta_g\end{array}) \left(\begin{array}{cc} A & B \\ 0 & (A\transp)^{-1} \end{array} \right) 
\end{align*}
where $b_g = (\omega_1, \ldots,\omega_g,\eta_1,\ldots,\eta_g)$ is the universal symplectic-Hodge basis of $\pi_{g,S}^*\mathcal{H}_{g,S}$.
\end{obs}

\begin{proof}[Proof of Proposition \ref{B=0}]
Since the vector bundle $\mathcal{R}_{g,S}$ is generated by the higher Ramanujan vector fields $v_{ij}$, we have $Dp(\mathcal{R}_{g,S})\subset p^*\mathcal{R}_{g,S}$ if and only if 
\begin{align*}
p_*v_{ij} \in \Gamma(\mathcal{B}_{g,S},\mathcal{R}_{g,S})
\end{align*}
for every $1\le i \le j \le g$. Further, by Remark \ref{remarkrg}, $p_*v_{ij}$ lies in  $\Gamma(\mathcal{B}_{g,S},\mathcal{R}_{g,S})$ if and only if
\begin{align*}
\nabla_{p_*v_{ij}}\eta_k = 0
\end{align*}
for every $1\le k \le g$. Finally, by the compatibility (\ref{comppbpf}), we conclude that $Dp(\mathcal{R}_{g,S})\subset p^*\mathcal{R}_{g,S}$ if and only if
\begin{align*}
\nabla_{v_{ij}}(p^*\eta_k) = 0
\end{align*}
for every $1\le i \le j \le g$ and $1\le k \le g$.

Now, by Remark \ref{h*carac}, we have
\begin{align*}
p^* \eta_k = \sum_{l=1}^g\omega_lB_{lk} + \eta_l(A^{-1})_{kl}\text{.}
\end{align*}
Using Proposition \ref{caracchamps}, we obtain
\begin{align*}
\nabla_{v_{ij}}(p^*\eta_k) = \sum_{l=1}^g(\nabla_{v_{ij}}\omega_l)B_{lk} = \begin{cases}
               \eta_iB_{ik} & i=j\\
               \eta_j B_{ik} + \eta_i B_{jk} & i< j 
              \end{cases}
\end{align*}
The assertion follows.
\end{proof}

Let $L_g$ be the subgroup scheme of $P_g$ given by
\begin{align*}
L_g(V) = \left.\left\{\left(\begin{array}{cc} A & 0 \\ 0 & (A\transp)^{-1} \end{array} \right) \in M_{2g\times 2g}(R_0) \ \right| A \in {\GL}_g(R_0) \right\}
\end{align*}
for any affine scheme $V=\Spec R_0$. The above proposition shows in particular that the action of $L_g$ on $\mathcal{B}_g$ preserves the integrable subbundle $\mathcal{R}_g$. The next proposition gives a precise transformation law for the higher Ramanujan vector fields $v_{ij}$ under the action of $L_g$.

\begin{prop}
Let $v = (v_{ij})_{1\le i,j\le g}$ be the unique symmetric matrix of global sections of $T_{\mathcal{B}_{g,S}/S}$ where $v_{ij}$ are the higher Ramanujan vector fields for $1\le i \le j \le g$. For every
\begin{align*}
p = \left(\begin{array}{cc} A & 0 \\ 0 & (A^{\mathsf{T}})^{-1} \end{array} \right) \in L_g(S)
\end{align*}
if we denote $p_*v= (p_*v_{ij})_{1\le i, j \le g}$, then we have the equality of matrices of sections $T_{\mathcal{B}_{g,S}/S}$ over $S$
\begin{align*}
p_*v = A v A^{\mathsf{T}}\text{,}
\end{align*}
\end{prop}

\begin{proof}
For each $1\le i \le j \le g$, put
\begin{align*}
w_{ij}\defeq \sum_{m,n=1}^gA_{im}v_{mn}A_{jn}\text{.}
\end{align*}
Then we must prove that $p_*v_{ij} = w_{ij}$ for every $1\le i \le j \le g$, which, by Lemma \ref{corocommute1}, is equivalent to proving that
\begin{align}\label{form1}
\nabla_{p_*v_{ij}}\omega_k = \nabla_{w_{ij}}\omega_k\text{, }\ \ \ \nabla_{p_*v_{ij}}\eta_k = \nabla_{w_{ij}}\eta_k
\end{align}
for every $1\le k \le g$. By compatibility (\ref{comppbpf}), equations (\ref{form1}) are equivalent to
\begin{align*}
\nabla_{v_{ij}}(p^*\omega_k) = p^*(\nabla_{w_{ij}}\omega_k)\text{, }\ \ \ \nabla_{v_{ij}}(p^*\eta_k) = p^*(\nabla_{w_{ij}}\eta_k)\text{.}
\end{align*}
As each $\eta_k$ is horizontal for the Gauss-Manin connection, we have $\nabla_{w_{ij}}\eta_k=0$. Further, as $p\in L_g(S)$, each $p^*\eta_k$ is an $R$-linear combination of $\eta_1,\ldots,\eta_g$; thus $p^*\eta_k$ is horizontal for the Gauss-Manin connection.

 We are thus reduced to proving that
\begin{align*}
\nabla_{v_{ij}}(p^*\omega_k) = p^*(\nabla_{w_{ij}}\omega_k)
\end{align*}
for every $1\le i \le j \le g$ and $1\le k \le g$. On the one hand, we have
\begin{align*}
p^*\omega_k = \sum_{l=1}^g\omega_lA_{lk}\text{,}
\end{align*}
so that, by Proposition \ref{caracchamps},
\begin{align*}
\nabla_{v_{ij}}(p^*\omega_k) = \sum_{l=1}^g(\nabla_{v_{ij}}\omega_l)A_{lk}= \eta_jA_{ik} + \eta_iA_{jk}\text{.}
\end{align*}
On the other hand,
\begin{align*}
\nabla_{w_{ij}}\omega_k &= \sum_{m,n=1}^g A_{im}(\nabla_{v_{mn}}\omega_k)A_{jn}\\
                      &= \sum_{n=1}^gA_{ik}\eta_nA_{jn} + \sum_{m=1}^gA_{im}\eta_mA_{jk}\\
                      &= A_{ik}\left(\sum_{n=1}^g\eta_nA_{jn} \right) + \left(\sum_{m=1}^gA_{im}\eta_m\right)A_{jk}\text{,}
\end{align*} 
hence, by Remark \ref{h*carac},
\begin{align*}
p^*(\nabla_{w_{ij}}\omega_k) = A_{ik}\left(\sum_{n=1}^gp^*\eta_n(A\transp)_{nj} \right) + \left(\sum_{m=1}^gA_{im}p^*\eta_m\right)A_{jk}= A_{ik}\eta_j + \eta_iA_{jk}\text{.}
\end{align*}
\end{proof}

\section{The case $g=1$: explicit equations} \label{caseg=1}

When $g=1$, we can compute explicit equations for $B_g$ and for the Ramanujan vector field.

\subsection{Explicit equation for the universal elliptic curve $X_1$ over $B_1$ and its universal symplectic-Hodge basis}

Fix a scheme $U$. Let us recall that every \emph{elliptic curve} $E$ over $U$ (namely, an abelian scheme of relative dimension 1) has a canonical unique principal polarization $\lambda_{E}:E \to E^t$ given, for any $U$-scheme $V$ and any point $P\in E(V)$, by 
\begin{align*}
\lambda_E(P)= \mathcal{O}_E([P]-[O])
\end{align*}
where $O\in E(V)$ denotes the identity section and $\mathcal{O}_E([P]-[O])$ denotes the class in $E^t(V)$ of the inverse of the ideal sheaf defined by the relative Cartier divisor $[P]-[O]$.

Therefore, the functor 
\begin{align*}
E \mapsto (E,\lambda_{E})
\end{align*}
defines an equivalence between the category of elliptic curves over $U$ and that of principally polarized elliptic curves over $U$. We can thus ``forget'' the principal polarization: an elliptic curve $E$ will always be assumed to be endowed with its canonical principal polarization $\lambda_{E}$. In particular, an object of $\mathcal{B}_1$ will be denoted simply by a ``couple'' $(E,b)_{/U}$.

\begin{obs}
The symplectic form induced by $\lambda_E$ coincides with the composition of the cup product in de Rham cohomology $H^1_{\dR}(E/U) \times H^1_{\dR}(E/U) \to H^2_{\dR}(E/U)$ with the trace map $H^2_{\dR}(E/U) \to \mathcal{O}_U$.  
\end{obs}

\begin{theorem} \label{casg=1}
Let 
\begin{align*}
B_1 \defeq \Spec \ZZ[1/2,b_2,b_4,b_6,\Delta^{-1}]
\end{align*}
where
\begin{align*}
\Delta \defeq \frac{b_2^2(b_4^2 - b_2b_6)}{4} - 8 b_4^3 - 27b_6^2 + 9b_2b_4b_6 = 16\, {\rm disc}\left(x^3 + \frac{b_2}{4}x^2 + \frac{b_4}{2}x + \frac{b_6}{4}\right)\text{,}
\end{align*}
and let $X_1$ be the elliptic curve over $B_1$ given by the equation
\begin{align*}
y^2 = x^3 + \frac{b_2}{4}x^2 + \frac{b_4}{2}x + \frac{b_6}{4}\text{.}
\end{align*}
Then $b_1 = (\omega_1,\eta_1)$ defined by 
\begin{align*}
\omega_1 \defeq \frac{dx}{2y}\text{, } \ \ \ \eta_1 \defeq x\frac{dx}{2y}
\end{align*}
is a symplectic-Hodge basis of ${X_1}_{/B_1}$ and the morphism $B_1 \to \mathcal{B}_1$ corresponding to $(X_1,b_1)_{/B_1}$ induces an isomorphism of $B_1$ with the $\ZZ[1/2]$-stack $\mathcal{B}_{1,\ZZ[1/2]}$.
\end{theorem}

In other words, if $(X_1,b_1)_{/B_1}$ is defined as above, then for any $\ZZ[1/2]$-scheme $U$, and any elliptic curve $E$ over $U$ endowed with a symplectic-Hodge basis $b$, there exists a unique morphism \linebreak $F_{/f}:E_{/U} \to {X_1}_{/B_1}$  in $\mathcal{A}_{1,\ZZ[1/2]}$ such that $F^*b_1=b$.

\begin{proof}
It is classical that $\omega_1$ so defined is in $F^1(X_1/B_1)$. To prove that $\langle \omega_1,\eta_1\rangle_{\lambda_E} = 1$ one can, for instance,  use the compatibility with base change to reduce this statement to an analogous statement concerning an elliptic curve over $\CC$, and then apply the classical residue formula (cf. \cite{DMOS82} pp. 23-25). 

Let $U$ be a $\ZZ[1/2]$-scheme and $(E,b)_{/U}$ be an object of $\mathcal{B}_{1}(U)$, with $b=(\omega,\eta)$. It is sufficient to prove that, locally for the Zariski topology over $U$, there exists a unique morphism $(E,b)_{/U} \to (X_1,b_1)_{/B_1}$ in $\mathcal{B}_{1,\ZZ[1/2]}$.

 We follow essentially the same steps in \cite{KM85} 2.2 to find a Weierstrass equation for an elliptic curve. Let us denote by $O: U \to E$ the identity section of the elliptic curve $E$ over $U$ and by $p:E \to U$ its structural morphism. Locally for the Zariski topology on $U$ we can find a formal parameter $t$ in the neighborhood of $O$ such that $\omega$ has a formal expansion in $t$ of the form
\begin{align*}
\omega = (1+ O(t))dt\text{,}
\end{align*} 
where $O(t)$ stands for a formal power series in $t$ of order $\ge 1$. Up to replacing $U$ by an open subscheme, we can and shall assume from now on that $t$ exists globally over $U$. 

There exist bases $(1,x)$ of $p_*\mathcal{O}_E(2 [O])$, and $(1,x,y)$ of $p_*\mathcal{O}_E(3[ O])$, such that
\begin{align} %\tag{$*$}
x = \frac{1}{t^2}(1+ O(t))\ \  \text{ and }\ \ y=\frac{1}{t^3}(1+O(t))\text{.}
\end{align}
Then the rational functions $x$ and $y$ necessarily satisfy an equation of the form
\begin{align*}
y^2 + a_1xy +a_3y= x^3 + a_2x^2 + a_4x+a_6\text{,}
\end{align*}
where $a_i$ are uniquely defined global sections of $\mathcal{O}_U$. Since 2 is invertible in $U$, the above equation is equivalent to
\begin{align*}
\left(y+\frac{a_1}{2}x + \frac{a_3}{2}\right)^2 = x^3 + \left(\frac{a_1^2 + 4 a_2}{4}\right)x^2 + \left(\frac{a_1a_3+2a_4}{2} \right)x + \frac{a_3^2 + 4 a_6}{4}\text{.}
\end{align*}
Therefore, after the change of coordinates $(x,y)\mapsto (x,y + \frac{a_1}{2}x + \frac{a_3}{2})$, we can assume that $x$ and $y$ satisfy
\begin{align*}
y^2 = x^3 + \frac{b_2}{4}x^2 + \frac{b_4}{2}x + \frac{b_6}{4}\text{,}
\end{align*}
where $b_i$ are global sections of $\mathcal{O}_U$. Put differently, we obtain a morphism $F_{/f} : E_{/U} \to {X_1}_{/B_1}$ in $\mathcal{A}_{1,\ZZ[1/2]}$.

By considering formal expansions in $t$, we see that $F^*\omega_1 = \omega$. In particular, 
\begin{align*}
(\omega, F^*\eta_1) = F^*b_1
\end{align*}
is a symplectic-Hodge basis of $E_{/U}$, and there exists a section $s$ of $\mathcal{O}_U$ such that $\eta = F^*\eta_1 + s\omega$. Thus, after the change of coordinates $(x,y)\mapsto (x+s,y)$, we have $F^*b_1 = b$. Therefore, we have constructed a morphism $F_{/f}:(E,b)_{/U} \to (X_1,b_1)_{/B_1}$ in $\mathcal{B}_{1,\ZZ[1/2]}$. 

We now prove that the morphism $F_{/f}$ is unique. Let $F'_{/f'} : (E,b)_{/U} \to (X_1,b_1)_{/B_1}$ be any morphism in $\mathcal{B}_{1,\ZZ[1/2]}$. If $f'=(b_2',b_4',b_6')$ are the coordinates of $f'$, then $F'$ is given by a basis $(1,x',y')$ of $p_*\mathcal{O}_E(3[ O])$ satisfying
\begin{align*} \tag{$*$}
(y')^2 = (x')^3 + \frac{b_2'}{4}(x')^2 + \frac{b_4'}{2}x' + \frac{b_6'}{4}\text{.}
\end{align*} 
As both $(1,x,y)$ and $(1,x',y')$ (resp. $(1,x)$ and $(1,x')$) are a basis of $p_*\mathcal{O}_E(3[ O])$ (resp. $p_*\mathcal{O}_E(2[O])$), then there exists global sections $c_1,c_2,c_3$ of $\mathcal{O}_{U}$ (resp. $u,v$ of $\mathcal{O}_U^{\times}$) such that
\begin{align*}
x' &= u(x+c_1)\\
y' &= v(y + c_2x+c_3)\text{.} 
\end{align*}
Note that equation $(*)$ implies that $u^3 = v^2$.

Now, as $(F')^*\omega_1 = F^*\omega_1$, we obtain
\begin{align*}
\frac{dx'}{2y'} = \frac{dx}{2y} \iff \frac{u}{v}\frac{dx}{2(y+c_2x+c_3)} = \frac{dx}{2y}\text{,}
\end{align*}
thus $c_2x+c_3=0$ and $u = v$. Since $u^3=v^2$, we obtain $u=v=1$ and $(x',y')=(x+c_1,y)$. Finally, as $(F')^*\eta_1=F^*\eta_1$, we have
\begin{align*}
x'\frac{dx'}{2y'} = \frac{dx}{2y} \iff  x\frac{dx}{2y} + c_1\frac{dx}{2y}= x\frac{dx}{2y}\text{,}
\end{align*}
hence $c_1=0$. Thus $(x',y')=(x,y)$ and this also implies that $f=f'$.
\end{proof}

\begin{obs} \label{classique1}
By considering the change of variables
\begin{align*}
\left\{\begin{array}{l}
         b_2 = e_2 \\
         b_4 = (e_2^2-e_4)/24\\
         b_6 = (4e_2^3 - 12 e_2e_4 + 8e_6)/1728
         \end{array}\right.\iff
\left\{\begin{array}{l}
         e_2 = b_2 \\
         e_4= b_2^2-24b_4\\
         e_6 = b_2^3 -36b_2b_4 + 216b_6
         \end{array}\right.
\end{align*}
we see that $B_1 \tensor_{\ZZ[1/2]}\ZZ[1/6]$ is isomorphic to
\begin{align*}
\Spec \ZZ[1/6, e_2,e_4,e_6,(e_4^3-e_6^2)^{-1}]\text{.}
\end{align*}
Under this identification, the universal elliptic curve $X_1$ is given by the equation
\begin{align*}
y^2 = 4\left(x+\frac{e_2}{12} \right)^3 -\frac{e_4}{12}\left(x+\frac{e_2}{12} \right) + \frac{e_6}{216}\text{,}
\end{align*}
and the universal symplectic-Hodge basis $b_1$ by $(dx/y,xdx/y)$. 
\end{obs}

\subsection{Explicit formula for the Ramanujan vector field} 

It is also possible to give an explicit formula for the Ramanujan vector field $v_{11}$ over $B_1$. Indeed, consider the global section of $T_{B_1/\ZZ[1/2]}$ given by\begin{align*}
v \defeq 2b_4\frac{\partial}{\partial b_2} + 3b_6\frac{\partial}{\partial b_4} + (b_2b_6 - b_4^2)\frac{\partial}{\partial b_6}\text{.}
\end{align*}
One may easily verify using the expression for the Gauss-Manin connection on $H^1_{\dR}(X_1/B_1)$ given in \ref{GMunivell} that
\begin{align*}
\nabla_v\left(\begin{array}{cc}
       \omega_1 & \eta_1
       \end{array}\right) =
       \left(\begin{array}{cc}
       \omega_1 & \eta_1
       \end{array}\right)\left(\begin{array}{cc}
       0 & 0\\
       1 & 0
       \end{array}\right)
\end{align*}
By Proposition \ref{caracchamps}, $v$ is the Ramanujan vector field $v_{11}$ over $B_1$.

\begin{obs} \label{explanation}
Under the isomorphism $B_1\tensor_{\ZZ[1/2]} \ZZ[1/6] \cong \ZZ[1/6,e_2,e_4,e_6,(e_4^3 - e_6^2)^{-1}]$ of Remark \ref{classique1}, $v$ gets identified with the vector field associated to the classical Ramanujan equations:
\begin{align*}
v = \frac{e_2^2-e_4}{12} \frac{\partial}{\partial e_2} + \frac{e_2e_4-e_6}{3}\frac{\partial}{\partial e_4} + \frac{e_2e_6-e_4^2}{2}\frac{\partial}{\partial e_6}\text{.}
\end{align*}
\end{obs}

\begin{appendices}

\section{Symplectic vector bundles} \label{sympl}

Fix once and for all a scheme $U$.

\subsection{Symplectic vector bundles}

 Let $\mathcal{E}$ a vector bundle over $U$. An $\mathcal{O}_U$-bilinear map
\begin{align*}
\langle \ , \ \rangle : \mathcal{E} \times \mathcal{E} \longrightarrow \mathcal{O}_U
\end{align*}
is said to be
\begin{enumerate}
    \item \emph{non-degenerate} if the $\mathcal{O}_U$-morphism $e \mapsto \langle \ \ , e \rangle$ from $\mathcal{E}$ to $\mathcal{E}^{\vee}$ is an isomorphism,
    \item \emph{alternating} if $\langle e, e \rangle =0$ for every section $e$ of $\mathcal{E}$.
\end{enumerate} 

\begin{defi} \label{defisymplbundle}
A \emph{symplectic form} over $\mathcal{E}$ is a non-degenerate and alternating $\mathcal{O}_U$-bilinear form over $\mathcal{E}$. A \emph{symplectic vector bundle} over $U$ is a couple $(\mathcal{E},\langle \ , \ \rangle)$, where $\mathcal{E}$ is a vector bundle over $U$ and $\langle \ , \ \rangle$ is a symplectic form over $\mathcal{E}$.
\end{defi}

% If $U=\Spec R$ is affine, we prefer to consider the module of global sections $E=\Gamma(U,\mathcal{E})$ in place of $\mathcal{E}$, and we shall also say that $E$ is a vector bundle over $R$.

\subsection{Lagrangian subbundles}

 Let $(\mathcal{E}, \langle \ , \ \rangle)$ be a symplectic vector bundle over $U$ and $\mathcal{F}$ be a subbundle of $\mathcal{E}$. We denote by $\mathcal{F}^{\langle \ , \ \rangle}$ the subsheaf of $\mathcal{E}$ consisting of those sections $e$ of $\mathcal{E}$ such that $\langle f , e \rangle = 0$ for every section $f$ of $\mathcal{F}$.

\begin{lemma} \label{exactseq}
We have an exact sequence of $\mathcal{O}_U$-modules
\begin{align*}
0 \longrightarrow  \mathcal{F}^{\langle \ , \ \rangle} \longrightarrow \mathcal{E} &\longrightarrow \mathcal{F}^{\vee} \longrightarrow 0\\
e &\mapsto\langle \ \ , e \rangle|_{\mathcal{F}}
\end{align*}
In particular, $\mathcal{F}^{\langle \ , \ \rangle}$ is a subbundle of $\mathcal{E}$ of rank ${\rm rank}( \mathcal{E}) - {\rm rank}(\mathcal{F})$.
\end{lemma}

\begin{proof}
The sequence $0 \longrightarrow  \mathcal{F}^{\langle \ , \ \rangle} \longrightarrow \mathcal{E} \longrightarrow \mathcal{F}^{\vee}$ is exact by definition. To see that $\mathcal{E} \to \mathcal{F}^{\vee}$ defined above is surjective, one may work locally and remark that in this case $\mathcal{F}$ is a direct factor of $\mathcal{E}$, and thus any $\mathcal{O}_U$-linear functional on $\mathcal{F}$ can be extended to $\mathcal{E}$; then one applies the non-degeneracy of the bilinear form $\langle \ , \ \rangle$.
\end{proof}

\begin{defi} \label{isolagr}
A subbundle $\mathcal{F}$ of $\mathcal{E}$ is said to be \emph{isotropic} with respect to $\langle \ , \ \rangle$ if $\mathcal{F} \subset \mathcal{F}^{\langle \ , \ \rangle}$. An isotropic subbundle of $\mathcal{E}$ such that $\mathcal{F}=\mathcal{F}^{\langle \ , \ \rangle}$ is said to be a \emph{Lagrangian subbundle}.
\end{defi}

The next result easily follows from Lemma \ref{exactseq}.

\begin{coro} \label{corosympl}
Let $\mathcal{F}$ be an isotropic subbundle of $\mathcal{E}$. Then $2 \rank( \mathcal{F}) \le \rank( \mathcal{E})$. Moreover, $\mathcal{F}$ is Lagrangian if and only if $2 \rank(\mathcal{F})=\rank(\mathcal{E})$.
\end{coro}

The next lemma shows that Lagrangian subbundles exist locally for the Zariski topology over $U$. This implies in particular that the rank of every symplectic vector bundle is even.

\begin{lemma} \label{locallagrangian}
Let $(\mathcal{E},\langle \ , \ \rangle)$ be a symplectic vector bundle over $U$ and assume that $U=\Spec R$, where $R$ is a local ring. Then there exists a Lagrangian subbundle of $\mathcal{E}$.
\end{lemma}

\begin{proof}
Let $S$ be the set of isotropic subbundles of $\mathcal{E}$ ordered by inclusion. It is sufficient to prove that every maximal element in $S$ is a Lagrangian (maximal elements always exist; consider the rank, for instance).

We proceed by contraposition. Let $\mathcal{F}$ be an element of $S$ that is not a Lagrangian. As $R$ is local and both $\mathcal{F}$ and $\mathcal{F}^{\langle \ , \ \rangle}$ are subbundles $\mathcal{E}$ (cf. Lemma \ref{exactseq}), there exists an integer $k\ge 1$ and global sections $e_1,\ldots,e_k$ of $\mathcal{F}^{\langle \ , \ \rangle}$ such that
\begin{align*}
\mathcal{F}^{\langle \ , \ \rangle} = \mathcal{F} \oplus \mathcal{O}_Ue_1 \oplus \cdots \oplus \mathcal{O}_Ue_k\text{.}
\end{align*}
 In particular, $\mathcal{F} \oplus \mathcal{O}_Ue_1$ is an element of $S$ strictly containing $\mathcal{F}$; thus, $\mathcal{F}$ is not maximal.
\end{proof}

\subsection{Symplectic bases} 

Let $(\mathcal{E}, \langle \ , \ \rangle)$ be a symplectic vector bundle of constant rank $2n$ over $U$.

\begin{defi} \label{defsymplbasis}
A \emph{symplectic basis} of $(\mathcal{E},\langle \ , \ \rangle)$ over $U$ is a basis of $\mathcal{E}$ over $U$ of the form $(e_1,\ldots,e_n,f_1,\ldots,f_n)$ with $\langle e_i, e_j \rangle = \langle f_i, f_j \rangle =0$ and $\langle e_i,f_j\rangle = \delta_{ij}$ for all $1\le i,j\le n$.
\end{defi}

 As Lagrangian subbundles exist locally by Lemma \ref{locallagrangian}, the next proposition implies in particular that symplectic bases also exist locally.

\begin{prop} \label{exisunic}
Let $U$ be an affine scheme, $(\mathcal{E},\langle \ , \ \rangle)$ be a symplectic vector bundle over $U$, and $\mathcal{L}$ be a Lagrangian subbundle of $\mathcal{E}$. Then
\begin{enumerate}
    \item Every basis $(e_1,\ldots,e_n)$ of $\mathcal{L}$ over $U$ can be completed to a symplectic basis $(e_1,\ldots,e_n,f_1,\ldots,f_n)$ of $\mathcal{E}$ over $U$.
    \item If $\mathcal{F}$ is a Lagrangian subbundle of $\mathcal{E}$ such that $\mathcal{L}\oplus \mathcal{F}=\mathcal{E}$, and $(f_1,\ldots,f_n)$ is a basis of $\mathcal{F}$ over $U$, then there exists a unique basis $(e_1,\ldots,e_n)$ of $\mathcal{L}$ over $U$ such that $(e_1,\ldots,e_n,f_1,\ldots,f_n)$ is a symplectic basis of $\mathcal{E}$ over $U$.
\end{enumerate}
\end{prop}

\begin{proof}
Consider the surjective morphism of $\mathcal{O}_U$-modules (cf. Lemma \ref{exactseq})
\begin{align*}
\mathcal{E} &\to \mathcal{L}^{\vee}\\
         e &\mapsto \langle \ \ , e\rangle|_{\mathcal{L}}\text{.}
\end{align*}
Since $U$ is affine, there exists a sequence $(f_1',\ldots,f_n')$ of global sections of $\mathcal{E}$ lifting the dual basis of $(e_1,\ldots,e_n)$ in $\mathcal{L}^{\vee}$, so that $\langle e_i,f_j'\rangle = \delta_{ij}$ for every $1\le i,j\le n$. As $\mathcal{L}$ is an isotropic subbundle of $\mathcal{E}$, to prove (1) it is sufficient to show the existence of global sections $\ell_j$ of $\mathcal{L}$ such that 
\begin{align*}
f_j \defeq f_j' + \ell_j
\end{align*}
satisfy $\langle f_i,f_j\rangle = 0$ for every $1\le i,j\le n$.

Since the bilinear form $\langle \ , \ \rangle$ is alternating, $A \defeq (\langle f_i',f_j'\rangle)_{1\le i,j\le n}$ is an antisymmetric matrix in $M_{n\times n}(\mathcal{O}_U(U))$. Thus, there exists a matrix $B = (b_{ij})_{1\le i,j \le g}$ in $M_{n\times n}(\mathcal{O}_U(U))$ such that $A = B - B\transp$. We put
\begin{align*}
\ell_i \defeq \sum_{j=1}^n b_{ij}e_j\text{,}
\end{align*}
hence
\begin{align*}
\langle f_i,f_j \rangle  = \langle f_i',f_j'\rangle + \langle \ell_i,f_j' \rangle - \langle \ell_j,f_i' \rangle = \langle f_i',f_j'\rangle + b_{ij} - b_{ji} = 0\text{.} 
\end{align*}

We now proceed to the proof of (2). As $\mathcal{F}$ is an isotropic subbundle of $\mathcal{E}$ satisfying $\mathcal{L} \oplus \mathcal{F} = \mathcal{E}$, the morphism of $\mathcal{O}_U$-modules
\begin{align*}
\mathcal{F} &\to \mathcal{L}^{\vee}\\
        f &\mapsto \langle \ \ , f\rangle|_{\mathcal{L}}  
\end{align*}
is injective by non-degeneracy of $\langle \ , \ \rangle$, thus an isomorphism since $\mathcal{F}$ and $\mathcal{L}^{\vee}$ have equal rank. The existence and unicity of $(e_1,\ldots,e_n)$ follows from remarking that $(e_1,\ldots,e_n, f_1,\ldots,f_n)$ is a symplectic basis of $\mathcal{E}$ over $U$ if and only if $(e_1,\ldots,e_n)$ is the basis of $\mathcal{L}$ over $U$ dual to the basis $(\langle \ \ , f_1\rangle|_{\mathcal{L}}, \ldots, \langle \ \ , f_n\rangle|_{\mathcal{L}})$ of $\mathcal{L}^{\vee}$.
\end{proof}

\section{Gauss-Manin connection on some elliptic curves}

\subsection{The Weierstrass elliptic curve}

Let
\begin{align*}
W \defeq \Spec \CC[g_2,g_3,\Delta^{-1}]
\end{align*}
where
\begin{align*}
\Delta \defeq g_2^3 - 27g_3^2\text{.}
\end{align*}
Then we can define an elliptic curve $E$ over $W$ by the classical Weierstrass equation
\begin{align*}
y^2 = 4x^3 -g_2x-g_3\text{.}
\end{align*}
Further, we define a symplectic-Hodge basis $(\omega,\eta)$ of $E_{/W}$ by the formulas
\begin{align*}
\omega \defeq \frac{dx}{y}\text{, } \ \ \ \eta \defeq x\frac{dx}{y}\text{.} 
\end{align*}

\begin{lemma} \label{GMconW}
With the above notations, the Gauss-Manin connection $\nabla$ on $H^1_{\dR}(E/W)$ is given by
\begin{align*}
\nabla \left(\begin{array}{cc}
       \omega & \eta
       \end{array}\right) = \left(\begin{array}{cc}
       \omega & \eta\end{array}\right)  \tensor \frac{1}{\Delta}
       \left(\begin{array}{cc}
       \Omega_{11} & \Omega_{12}\\
       \Omega_{21} & \Omega_{22}
       \end{array}\right)
\end{align*}
where
\begin{align*}
\Omega_{11} &= -\frac{1}{4}g_2^2\, dg_2 + \frac{9}{2}g_3\,dg_3\\
\Omega_{12} &= \frac{3}{8}g_2g_3\,dg_2 -\frac{1}{4}g_2^2\, dg_3\\
\Omega_{21} &= -\frac{9}{2}g_3 \, dg_2 + 3g_2\, dg_3\\
\Omega_{22} &= -\Omega_{11}\text{.}
\end{align*}
\end{lemma}

Let us briefly explain how these expressions follow from the description given in \cite{katz73} A1.3 of the Gauss-Manin connection on the relative first de Rham cohomology of the universal elliptic curve $\mathbb{E}$ over the Poincaré half-plane $\mathbb{H}$ (whose fiber at each $\tau\in \mathbb{H}$ is given by the complex torus $\mathbb{E}_{\tau}=\CC/(\ZZ + \ZZ\tau)$).\footnote{A direct algebraic approach is also possible. See for instance \cite{KO68} 3 and \cite{kedlaya07} 3.4.}

 We first remark that for any $u\in \CC^{\times}$ we can define an automorphism ${M_u}_{/\mu_u}: E_{/W} \to E_{/W}$ in the category $\mathcal{A}_{1,\CC}$ by
\begin{align*}
\mu_u(g_2,g_3) = (u^{-4}g_2,u^{-6}g_3)\text{, } \ \ \ M_u(x,y)=(u^{-2}x,u^{-3}y)\text{.}
\end{align*}
Using that the Gauss-Manin connection commutes with base change and admits regular singularities, we deduce by homogeneity that there exists constants $c_1,\ldots,c_8$ in $\CC$ such that
\begin{align*}
\Omega_{11} = c_1g_2^2\, dg_2 + c_2g_3\, dg_3\text{, }\ \ &\Omega_{12} = c_3g_2g_3\, dg_2 +c_4g_2^2\, dg_3\text{, }\\
 \Omega_{21} = c_5g_3\, dg_2 + c_6g_2\, dg_3\text{, } \ \ & \Omega_{22} = c_7g_2^2\, dg_2 + c_8g_3\, dg_3 \text{.}
\end{align*}
To determine these constants, we consider the cartesian diagram in the category of complex analytic spaces
$$
\raisebox{-0.5\height}{\includegraphics{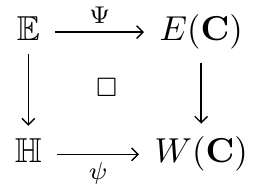}}
$$
% $$
% \begin{tikzcd}
% \mathbb{E} \arrow{d} \arrow{r}{\Psi} & E(\CC)\arrow{d} \\
% \mathbb{H} \arrow{r}[swap]{\psi} & W(\CC)\arrow[draw=none]{ul}[description]{\square}
% \end{tikzcd}
% $$
given by the classical Weierstrass theory:
\begin{align*}
\psi(\tau)=(g_2(\tau),g_3(\tau))\text{, }\ \ \ \Psi_{\tau}(z)=(\wp_{\tau}(z),\wp_{\tau}'(z)) 
\end{align*}
Finally, we apply once again that that the formation of the Gauss-Manin connection (now in the complex analytic category) commutes with base change, and we use the formulas in \cite{katz73} A1.3:
\begin{align*}
\nabla \left(\begin{array}{cc}
       dz & \wp_{\tau}(z)dz
       \end{array}\right) = \left(\begin{array}{cc}
       dz & \wp_{\tau}(z)dz\end{array}\right)  \tensor \frac{1}{2\pi i}
       \left(\begin{array}{cc}
       -(2\pi i)^2E_{2}(\tau)/12 & -(2\pi i)^4E_4(\tau)/144\\
       1 & (2\pi i)^2E_2(\tau)/12
       \end{array}\right)d\tau
\end{align*}

\subsection{The elliptic curve ${X}_{/B}$ over $\ZZ[1/6]$} \label{GMconB}

Let
\begin{align*}
B \defeq \Spec \ZZ[1/6, e_2,e_4,e_6, \Delta^{-1}]
\end{align*}
where
\begin{align*}
\Delta \defeq e_4^3 - e_6^2\text{.}
\end{align*}
We define an elliptic curve $X$ over $B$ by
\begin{align*}
y^2 = 4\left(x+ \frac{e_2}{12}\right)^3 - \frac{e_4}{12}\left(x+ \frac{e_2}{12}\right) + \frac{e_6}{216}\text{.}
\end{align*}
We define a symplectic-Hodge basis $(\omega, \eta)$ of $X_{/B}$ by the formulas
\begin{align*}
\omega \defeq \frac{dx}{y}\text{, } \ \ \ \eta \defeq x\frac{dx}{y}\text{.} 
\end{align*}

Note that there is a morphism $F_{/f} : (X_{\CC})_{/B_{\CC}} \to E_{/W}$ in $\mathcal{A}_{1,\CC}$ given by
\begin{align*}
f(e_2,e_4,e_6) = \left(\frac{e_4}{12},-\frac{e_6}{216}\right)\text{, } \ \ \ F(x,y) = \left(x + \frac{e_2}{12},y \right)\text{.}
\end{align*}
By pulling back the Gauss-Manin connection on $H^1_{\dR}(E/W)$ described in Lemma \ref{GMconW} by the morphism $F_{/f}$, we obtain that the Gauss-Manin connection $\nabla$ on $H^1_{\dR}(X/B)$ over $\ZZ[1/6]$ is given by
\begin{align*}
\nabla \left(\begin{array}{cc}
       \omega & \eta
       \end{array}\right) = \left(\begin{array}{cc}
       \omega & \eta\end{array}\right)  \tensor \frac{1}{\Delta}
       \left(\begin{array}{cc}
       \Omega_{11} & \Omega_{12}\\
       \Omega_{21} & \Omega_{22}
       \end{array}\right)
\end{align*}
where
\begin{align*}
\Omega_{11} &= \left(\frac{e_2e_6 - e_4^2}{4} \right)de_4 + \left(\frac{e_6-e_2e_4}{6} \right)de_6\\
\Omega_{12} &=-\frac{\Delta}{12}de_2 - \left(\frac{e_4e_6-2e_2e_4^2+e_2^2e_6}{48} \right) de_4 + \left(\frac{e_4^2-2e_2e_6 + e_2^2e_4}{72}\right)de_6\\
\Omega_{21} &= 3e_6de_4 -2e_4de_6\\
\Omega_{22} &= -\Omega_{11}\text{.}
\end{align*}

\subsection{The universal elliptic curve ${X_1}_{/B_1}$ over $\ZZ[1/2]$} \label{GMunivell}

Consider the elliptic curve $X_{1}$ over $B_1$ defined in Theorem \ref{casg=1} and let $\Phi_{/\varphi} : (X_{1,\ZZ[1/6]})_{/B_{1,\ZZ[1/6]}} \to X_{/B}$ be the isomorphism in $\mathcal{A}_{1,\ZZ[1/6]}$ given by
\begin{align*}
\varphi(b_2,b_4,b_6) = (b_2,b_2^2-24b_4,b_2^3 -36b_2b_4 + 216b_6)\text{, } \ \ \ \Phi(x,y)=(x,2y)\text{.} 
\end{align*}

If $(\omega_1,\eta_1)$ denotes de symplectic-Hodge basis of ${X_{1}}_{/B_1}$ defined in Theorem \ref{casg=1}, then by pulling back the Gauss-Manin connection on $H^1_{\dR}(X/B)$ described in \ref{GMconB} by the isomorphism $\Phi_{/\varphi}$, we obtain that the Gauss-Manin connection $\nabla$ on $H^1_{\dR}(X_{1}/B_1)$ over $\ZZ[1/2]$ is given by

\begin{align*}
\nabla \left(\begin{array}{cc}
       \omega_1 & \eta_1
       \end{array}\right) = \left(\begin{array}{cc}
       \omega_1 & \eta_1\end{array}\right)  \tensor \frac{1}{\Delta}
       \left(\begin{array}{cc}
       \Omega_{11} & \Omega_{12}\\
       \Omega_{21} & \Omega_{22}
       \end{array}\right)
\end{align*}
where
\begin{align*}
\Omega_{11} &= \frac{b_2^2b_6-6b_4b_6-b_2b_4^2}{8} db_2 + \frac{4b_4^2-3b_2b_6}{2}db_6 + \frac{18b_6 - b_2b_4}{4}db_6\\
\Omega_{12} &= \frac{2b_4^3 +9b_6^2 - 2b_2b_4b_6}{4}db_2 + \frac{b_2^2b_6-b_2b_4^2 -6b_4b_6}{4}db_6 + \frac{4b_4^2-3b_2b_6}{4}db_6\\
\Omega_{21} &= \frac{3b_2b_6 - 4b_4^2}{4}db_2 + \frac{b_2b_4-18b_6}{2}db_4 + \frac{24b_4-b_2^2}{4}db_6\\
\Omega_{22} &= -\Omega_{11}\text{.}
\end{align*}

\end{appendices}


\begin{thebibliography}{9}

% \bibitem{andre04}
%       Y. André,
%       \emph{Une introduction aux Motifs (Motifs Purs, Motifs Mixtes, Périodes)}.
%       Panoramas et Synthèses \textbf{17}. Société Mathématique de France, 2004.    

% \bibitem{borel60}
%       A. Borel,
%       \emph{Density properties for certain subgroups of semi-simple groups without compact components}.
%       Annals of Mathematics, Second Series, \textbf{72} No. 1 (1960) 179-188.

% \bibitem{BL04}
%       C. Birkenhake, H. Lange,
%       \emph{Complex Abelian Varieties (second, augmented edition)}.
%       Grundlehren der mathematischen Wissenschaften, Springer-Verlarg Berlin Heildelberg (2004).

% \bibitem{BN06}
%      K. Behrend and B. Noohi,
%      \emph{Uniformization of Deligne-Mumford analytic curves}.
%      J. Reine Angew. Math., \textbf{599} (2016), pp. 111-153.

\bibitem{BBM82}
      P. Berthelot, L. Breen, W. Messing,
      \emph{Théorie de Dieudonné cristalline II}. Lecture Notes in Mathematics \textbf{930}. Springer-Verlag, Berlin (1982).

\bibitem{BZ03}
        D. Bertrand, V. V. Zudilin,
        \emph{On the transcendence degree of the differential field generated by Siegel modular forms}.
        J. Reine Angew. Math. \textbf{554} (2003), 47-68.

\bibitem{BLR90}
        S. Bosch, W. Lütkebohmert, M. Raynaud,
        \emph{Néron Models}.
         Ergebnisse der Mathematik und ihrer Grenzgebiete \textbf{21} (1990), Springer-Verlag.   

% \bibitem{bost01}
%        J.-B. Bost,
%        \emph{Algebraic leaves of algebraic foliations over number fields},
%        Publ. Math. I.H.E.S. \textbf{93} (2001), 161-221.

\bibitem{chazy11}
      J. Chazy,
      \emph{Sur les équations différentielles du troisième ordre et d'ordre supérieur dont l'intégrale générale a des points critiques fixes}.
      Acta Math. \textbf{34} (1911), 317-385.

% \bibitem{debarre99}
%           O. Debarre,
%           \emph{Tores et variétés abéliennes complexes}.
%           Cours spécialisés (Collection SMF) \textbf{6} (1999).
%           Société Mathématique de France, EDP Sciences.

% \bibitem{cartan53}
%       H. Cartan,
%       \emph{Quotient d'une variété analytique par un groupe discret d'automorphismes}.
%       Séminaire Henri Cartan, \textbf{6} (1953-1954), Exp. No. 12, 13 p.

% \bibitem{deligne75}
%           P. Deligne,
%           \emph{Courbes elliptiques : formulaire (d'après J. Tate mis au goût du jour par P. Deligne)}.
%           Lecture Notes in Mathematics \textbf{476} (1975), 53-73.
%           Springer International Publishing.

% \bibitem{chudnovsky80}
%       G. V. Chudnovsky,
%       \emph{Algebraic independence of values of exponential and elliptic functions}. Proceedings of the international congress of mathematicians (Helsinki 1978), Academia Scientiarum Fennica, Helsinki (1980), 339-350.


% \bibitem{deligne70}
%       P. Deligne,
%       \emph{Equations Différentielles à Points Réguliers Singuliers}. 
%       Lecture Notes in Mathematics \textbf{163},
%       Springer-Verlag, Berlin (1970).

% \bibitem{deligne71}
%       P. Deligne,
%       \emph{Théorie de Hodge : II}.
%       Publications mathématiques de l'IHÉS \textbf{40} (1971), p. 5-57. 

% \bibitem{DM69}
%      P. Deligne, D. Mumford,
%      \emph{The irreducibility of the space of curves of given genus}.
%      Publications Mathématiques de l'IHÉS \textbf{36} (1969), p. 75-109.

% \bibitem{deligne80}
%       P. Deligne (redigé par J.L. Brylinsky),
%       \emph{Cycles de Hodge absolus et périodes des intégrales des variétés abéliennes.} Mémoires SMF \textbf{2} (1980), 23-33.

\bibitem{DMOS82}
     P. Deligne, J. S. Milne, A. Ogus, K. Shih,
     \emph{Hodge Cycles, Motives, and Shimura Varieties}.
     Lecture Notes in Mathematics \textbf{900} (1982), Springer-Verlag. 

\bibitem{DP94}
      P. Deligne, G. Pappas,
      \emph{Singularités des espaces de modules de Hilbert, en les caractéristiques divisant le discriminant}.
      Compositio Mathematica, tome 90, no. 1 (1994), p. 59-79.

% \bibitem{demazure72}
%        M. Demazure,
%        \emph{Lectures on $p$-divisible groups}.
%        Lecture Notes in Mathematics \textbf{302},
%        Springer-Verlag, Berlin (1972). 

% \bibitem{EMV}
%       B. Edixhoven, G. van der Geer, B. Moonen,
%       \emph{Abelian Varieties (preliminary version of the first chapters)}.
%       Disponible en ligne sur \url{http://gerard.vdgeer.net/AV.pdf} (19/07/2016).


\bibitem{FC90}
         G. Faltings, C-L. Chai,
         \emph{Degeneration of abelian varieties}.
         Ergebnisse der Mathematik und ihrer Grenzgebiete \textbf{22} (1990).
         Springer-Verlag Berlin Heildeberg.

% \bibitem{grothendieck66}
%        A. Grothendieck,
%        \emph{On the de Rham cohomology of algebraic varieties}.
%        Publications mathématiques de l'IHES, tome 29 (1966), p. 95-103.

\bibitem{EGAIV3}
          A. Grothendieck (rédigé avec la collaboration de J. Dieudonné),
          \emph{Éléments de Géométrie algébrique IV. Étude locale des schémas et des morphismes de schémas, Troisième Partie}.
          Publications mathématiques de l'IHES, tome 32 (1967), p. 5-361.

\bibitem{EGAIV4}
          A. Grothendieck (rédigé avec la collaboration de J. Dieudonné),
          \emph{Éléments de Géométrie algébrique IV. Étude locale des schémas et des morphismes de schémas, Quatrième Partie}.
          Publications mathématiques de l'IHES, tome 32 (1967), p. 5-361.

\bibitem{SGA103}
         A. Grothendieck,
         \emph{Séminaire de Géométrie Algébrique du Bois Marie 1960–61 Revêtements étales et groupe fondamental (SGA 1)}.
         Augmenté de deux exposés de Mme M. Raynaud. 
         Documents Mathématiques \textbf{3} (2003), Société Mathématique de France.         

\bibitem{halphen81}
         M. Halphen,
         \emph{Sur un système d'équations différentielles}.
         C. R. Acad. Sci. Paris \textbf{92} (1881), 1101-1103.

\bibitem{jacobi48}
         C. G. J. Jacobi,
         \emph{Über die Differentialgleichung, welcher die Reihen $1\pm q + 2q^4 \pm 2q^9 +$ etc., $2\sqrt[4]{q} + 2\sqrt[4]{q^9} + 2\sqrt[4]{q^{25}} +$ etc. Genüge leisten}.
         Crelle's J. für Math. \textbf{36}:2 (1848), 97-112.

\bibitem{katz70}
         N. M. Katz,
         \emph{Nilpotent connections and the monodromy theorem: applications of a result of Turrittin}.
          Publications mathématiques de l'IHES, tome 39 (1970), p. 175-232.

\bibitem{katz73}
          N. M. Katz,
          \emph{$p$-adic properties of modular schemes and modular forms}.
          Lecture Notes in Mathematics \textbf{350} (1973), 69-190.
          Springer International Publishing.

\bibitem{KM85}
         N. M. Katz, B. Mazur,
         \emph{Arithmetic moduli of elliptic curves}.
         Annals of Mathematics Studies \textbf{108} (1985).
         Princeton University Press. 

\bibitem{KO68}
         N. M. Katz, T. Oda,
         \emph{On the differentiation of de Rham cohomology classes with respect to parameters}.
         Kyoto Journal of Mathematics \textbf{8}-2 (1968) 199-213.

\bibitem{kedlaya07}
        K. S. Kedlaya,
        \emph{$p$-adic cohomology: from theory to practice}.
        In \emph{$p$-adic Geometry, Lectures from the 2007 Arizona Winter School} (D. Savitt \& D. S. Thakur Eds.).
        University Lecture Series \textbf{45} (2008), American Mathematical Society.

% \bibitem{KM97}
%          S. Keel, S. Mori,
%          \emph{Quotients by groupoids}.
%          Annals of Mathematics \textbf{145} (1997), p. 191-224.

\bibitem{knutson71}
        D. Knutson,
        \emph{Algebraic spaces}.
        Lecture Notes in Mathematics \textbf{203} (1971).
        Springer-Verlag Berlin-Heildeberg-New York.

% \bibitem{LMB91}
%          G. Laumon, L. Moret-Bailly,
%          \emph{Champs algébriques}.
%          Ergebnisse der Mathematik und ihrer Grenzgebiete \textbf{39} (1991).
%          Springer-Verlag Berlin-Heildeberg.

% \bibitem{mahler69}
%         K. Mahler,
%         \emph{On algebraic differential equations satisfied by automorphic functions}.
%         Journal of Australian Mathematical Society \textbf{10} (1969), 445-450.

% \bibitem{milne86}
%          J. S. Milne,
%          \emph{Abelian varieties}.
%          In \emph{Arithmetic geometry} (G. Cornell \& J. H. Silverman Eds.).
%          Springer-Verlag New York (1986).

% \bibitem{moret-bailly85}
%         L. Moret-Bailly,
%         \emph{Pinceaux de variétés abéliennes}.
%         Astérisque \textbf{129} (1985). Société Mathématique de France.

\bibitem{movasati08}
        H. Movasati,
        \emph{On differential modular forms and some analytic relations between Eisenstein series.}
        Ramanujan Journal, \textbf{17} no. 1 (2008), 53-76.

\bibitem{movasati12}
        H. Movasati,
        \emph{Quasi-modular forms attached to elliptic curves, I}.
        Annales mathématiques Blaise Pascal, \textbf{19} no. 2 (2012), pp. 307-377.

% \bibitem{mumford67}
%         D. Mumford,
%         \emph{On the equations defining abelian varieties. II}.
%         Inventiones mathematicae, volume 3, issue 2, pp. 75-135, 1967.
       
\bibitem{mumford70}
        D. Mumford,
        \emph{Abelian varieties}.
        Tata institute of fundamental research, Bombay. Oxford University Press (1970).

% \bibitem{noohi05}
%         B. Noohi,
%         \emph{Foundations of Topological Stacks I}.
%         Preprint, available  at \url{https://arxiv.org/pdf/math/0503247v1.pdf} (2005). 

% \bibitem{NO80}
%         P. Norman and F. Oort,
%         \emph{Moduli of Abelian Varieties}.
%         Annals of Mathematics, Second Series, Vol. 112, No.2 (1980), pp. 413-439.

\bibitem{GIT94}
        D. Mumford, D. Fogarty, J. Kirwan,
        \emph{Geometric invariant theory}.
        Ergebnisse der Mathematik und ihrer Grenzgebiete \textbf{34} (1994).
        Springer-Verlag Berlin Heidelberg.

\bibitem{nesterenko96}
        Y. V. Nesterenko,
        \emph{Modular functions and transcendence questions}, Sb. Math. 187 1319, 1996.

\bibitem{oda69}
        T. Oda,
        \emph{The first de Rham cohomology group and Dieudonné modules}.
        Annales scientifiques de l'E.N.S. 4e série, tome 2, no 1 (1969), p. 63-135.

\bibitem{ohyama95}
        Y. Ohyama,
        \emph{Differential relations of theta functions},
        Osaka J. Math. \textbf{32} (1995), 431-450. 

\bibitem{ohyama96}
        Y. Ohyama,
        \emph{Differential equations of theta constants of genus two},
        Algebraic analysis of singular perturbations, Kyoto Univ. (1996), pp. 96-103. 

\bibitem{olsson12}
       M. Olsson,
       \emph{Compactifications of moduli of abelian varieties: An introduction}.
       ``Current Developments in Algebraic Geometry", 295–348, Math. Sci. Res. Inst. Publ., \textbf{59} (2012).
        Cambridge Univ. Press, Cambridge. 

\bibitem{olsson16}
       M. Olsson,
       \emph{Algebraic spaces and stacks}.
       American Mathematical Society Colloquium Publications, \textbf{62} (2016).

\bibitem{oort71}
        F. Oort,
        \emph{Finite group schemes, local moduli for abelian varieties, and lifting problems}.
        Compositio Mathematica, tome 23, no. 3 (1971), p. 265-296.

\bibitem{ramanujan16}
       S. Ramanujan,
       \emph{On certain arithmetical functions}.
       Trans. Cambridge Philos. Soc. \textbf{22} (1916), 159-184.

\bibitem{romagny05}
        M. Romagny,
        \emph{Group actions on stacks and applications}.
        Michigan Mathematical Journal \textbf{53} (2005).

% \bibitem{stacks}
%        The Stacks Project Authors, \emph{Stacks Project}.
%        \url{http://stacks.math.columbia.edu} (2016).

% \bibitem{zagier08}
%        D. Zagier,
%        \emph{Elliptic modular forms and their applications}.
%        In \emph{The 1-2-3 of Modular forms: Lectures at a Summer School in Nordfjordeid, Norway}. Universitext, Springer-Verglag, Berlin-Heidelberg-New York (20018), pp. 1-103.

\bibitem{serre72}
        J.-P. Serre,
        \emph{Congruences et formes modulaires}.
        Séminaire N. Bourbaki, 1971-1972, exp. no. 416, p. 319-338.

\bibitem{zudilin00}
        V. V. Zudilin,
        \emph{Thetanulls and differential equations} (Russian).
        Matematicheskii Sbornik \textbf{191}:12 (2000), 77-122; translation in Sb. Math. \textbf{191} (2000), no. 11-12, 1827-1871.

\end{thebibliography}
\end{document}